\documentclass[12pt,oneside]{amsart}
\usepackage[latin9]{inputenc}
\usepackage{amsthm}
\usepackage{amstext}
\usepackage{amssymb}
\usepackage{wasysym}
\usepackage{esint}

\makeatletter
\numberwithin{equation}{section}
\numberwithin{figure}{section}
\theoremstyle{plain}
\newtheorem{thm}{\protect\theoremname}
  \theoremstyle{definition}
  \newtheorem{defn}[thm]{\protect\definitionname}
  \theoremstyle{plain}
  \newtheorem{lem}[thm]{\protect\lemmaname}
  \theoremstyle{plain}
  \newtheorem{cor}[thm]{\protect\corollaryname}
  \theoremstyle{plain}
  \newtheorem{prop}[thm]{\protect\propositionname}
  \theoremstyle{remark}
  \newtheorem{rem}[thm]{\protect\remarkname}

\numberwithin{thm}{section}
\usepackage{pgf,tikz}
\usepackage{mathrsfs}
\usetikzlibrary{arrows}

\makeatother

  \providecommand{\corollaryname}{Corollary}
  \providecommand{\definitionname}{Definition}
  \providecommand{\lemmaname}{Lemma}
  \providecommand{\propositionname}{Proposition}
  \providecommand{\remarkname}{Remark}
\providecommand{\theoremname}{Theorem}

\begin{document}

\author{Thomas Dauer and Marlies Gerber}

\title{Generic Absence of Finite Blocking for Interior Points of Birkhoff
Billiards}

\thanks{The authors were supported in part by NSF grant DMS-1156515.}


\begin{abstract}
Let $x$ and $y$ be points in a billiard table $M=M(\sigma)$ in
$\mathbb{\mathbb{R}}^{2}$ that is bounded by a curve $\sigma$. We
assume $\sigma\in\Sigma_{r}$ with $r\geq2$, where $\Sigma_{r}$
is the set of simple closed $C^{r}$ curves in $\mathbb{R}^{2}$ with
positive curvature. A subset $B$ of $M\setminus\{x,y\}$ is called
a \textit{blocking set} for the pair $(x,y)$ if every billiard path
in $M$ from $x$ to $y$ passes through a point in $B$. If a finite
blocking set exists, the pair $(x,y)$ is called ${\it secure}$ in
$M;$ if not, it is called ${\it insecure}$. We show that for $\sigma$
in a dense $G_{\delta}$ subset of $\Sigma_{r}$ with the $C^{r}$
topology, there exists a dense $G_{\delta}$ subset $\mathcal{\mathcal{R}=R}(\sigma)$
of $M(\sigma)\times M(\sigma)$ such that $(x,y)$ is insecure in
$M(\sigma)$ for each $(x,y)\in\mathcal{R}$. In this sense, for the
generic Birkhoff billiard, the generic pair of interior points is
insecure. This is related to a theorem of S. Tabachnikov \cite{Tab09}
that $(x,y)$ is insecure for ${\it all}$ sufficiently close points
$x$ and $y$ on a strictly convex arc on the boundary of a smooth
table. 
\end{abstract}

\maketitle

\section{Introduction}

\markboth{Running head}{Absence of Finite Blocking in Birkhoff Billiards}Consider
a compact plane region $M=M(\sigma)$ (``table'') bounded by a simple
closed curve $\sigma.$ A \textit{billiard} is the dynamical system
consisting of this table and a point mass that moves within the table
at unit speed along line segments whose endpoints are on the boundary.
At the boundary, the direction of motion changes instantaneously,
so that the angle of incidence equals the angle of reflection. The
trajectory of the point mass during a given time interval is called
a ${\it billiard}\ {\it path}.$ We say that a set $B\subset M\setminus\{x,y\}$
is a blocking set for a pair of points $(x,y)\in M\times M$ if every
billiard path from $x$ to $y$ passes through a point of $B.$ If
a finite blocking set exists, the pair $(x,y)$ is called ${\it secure;}$
if not, it is called \textit{insecure}. For example, if the boundary
of $M$ is an ellipse with foci $x$ and $y,$ then $(x,y)$ is insecure.
In particular, if the boundary is a circle with center $C,$ the pair
$(C,C)$ is insecure, but for $z\in M$ with $z\neq C,$ $(C,z)$
is secure. A table is called secure if ${\it each}$ pair of points
in the table is secure; if not, the table is called insecure. 

It is an open problem to characterize secure billiard tables, even
in the case of polygonal tables (see the survey \cite{Gut12} by E.
Gutkin). In 2004, T. Monteil \cite{Mon04} showed that there exists
a rational billiard table (i.e., a polygonal billiard table in which
all angles are rational multiples of $\pi$) that is insecure, contradicting
some earlier work in this area. According to Gutkin \cite{Gut05},
for a regular $n$-gon to be secure, it is necessary and sufficient
that $n=3,4,$ or $6$. The proof of necessity is deep and is based
on earlier work on security in translation surfaces \cite{Gut84,Vee89,Vee92,GutJud96,GutJud00,GutHubSch}.
Moreover, Gutkin \cite{Gut05} proved the security of polygons that
are tiled under reflection by one of the following: a triangle with
angles of $30^{\circ},60^{\circ},90^{\circ},$ a triangle with angles
of $45^{\circ},45^{\circ},90^{\circ},$ an equilateral triangle, or
a rectangle. Gutkin's conjecture \cite{Gut12} that this is a necessary
condition for a polygonal table to be secure is still open. 

We consider Birkhoff billiard tables, which are defined to be convex
tables with a smooth boundary (at least $C^{2}$ in the context of
this paper). The only prior result regarding security for points in
such tables seems to be the one by S. Tabachnikov \cite{Tab09}, who
showed that for every compact billiard table $M$ bounded by a smooth
curve, in particular a Birkhoff billiard, if $x$ and $y$ are sufficiently
close distinct points on a strictly convex arc of $\partial M,$ then
$(x,y)$ is insecure. We consider the case of points $x$ and $y$
in the interior of a Birkhoff billiard table. While we do not obtain
information about the security of any specific pair of points, we
show that insecurity is ${\it generic}$ in the sense of Baire category,
that is, it holds on a residual set (which is a set whose complement
is meager). More precisely, for $r\geq2,$ let $\Sigma_{r}$ be the
set of simple closed $C^{r}$ curves in the plane that have positive
curvature (as defined in Section 3), with the $C^{r}$ topology. For
any given pair of distinct points $x$ and $y$ in the plane, let
$\Sigma_{r,(x,y)}$ be the set of curves in $\Sigma_{r}$ that enclose
a region containing $x$ and $y$ in its interior. We show that there
exists a residual subset $A_{r,(x,y)}$ of $\Sigma_{r,(x,y)}$ such
that for every billiard table bounded by a curve in $A_{r,(x,y)}$
the pair $(x,y)$ is insecure. The Kuratowski-Ulam Theorem allows
us to reformulate this as follows: There is a residual subset $A_{r}$
of $\Sigma_{r}$ such that for $\sigma\in A_{r},$ insecurity holds
for $(x,y)$ in a residual subset of $M(\sigma)\times M(\sigma)$.
(See Theorem \ref{thm:main result}.) For Birkhoff billiards we should
allow boundary curves of nonnegative curvature, but for our result
we may assume positive curvature, since $\Sigma_{r}$ is an open dense
subset of the set of simple closed $C^{r}$ curves of nonnegative
curvature with the $C^{r}$ topology. Thus, our result shows that
for the generic Birkhoff billiard, the generic pair of interior points
is insecure.

In the context of compact Riemannian manifolds, the security property
(defined by replacing ``billiard paths'' by ``geodesics'') has been
studied extensively (see section 5.6 of \cite{Gut12} for a summary
of these results and references). According to \cite{Bur08}, it is
expected that most Riemannian manifolds are ${\it totally\ insecure}$,
that is, all pairs $(x,y)$ are insecure. Total insecurity has been
established in various settings \cite{GutSch,LafSch,Bur08,Ban10,Ger13}.
M. Gerber and W.-K. Ku \cite{Ger11} proved that for any compact $C^{\infty}$
manifold, there is a residual set of metrics for which the set of
insecure pairs $(x,y)$ is residual. In the case of dimension two,
V. Bangert and Gutkin \cite{Ban10} proved that the following stronger
result: All Riemannian manifolds of genus greater than one are totally
insecure, and for genus one, a residual set of Riemannian metrics
is totally insecure. (See also the work of W. Ho \cite{Ho08} for
related results.) Even though the result for Birkhoff billiards in
the present paper is analogous to the result for Riemannian manifolds
in \cite{Ger11}, the techniques are quite different, because Riemannian
metrics can be modified anywhere within the manifold, while our modifications
can only change the boundary of the table, not the geometry within
the table.

\section{Outline of Our Approach}

Let $r\ge2,$ and let $M=M(\sigma)$ be a billiard table in $\mathbb{R}^{2}$
bounded by a curve $\sigma\in\mbox{\ensuremath{\Sigma}}_{r},$ as
in Section 1. We consider ${\it distinct}$ points $x$ and $y$ in
the interior of $M$. A ${\it vertex}$ of a billiard path from $x$
to $y$ is a point on this path that lies on the boundary of $M.$
(See Section 3 for a precise definition of billiard paths.) To prove
that the pair $(x,y)$ is insecure, it suffices to show that for every
positive integer $n$, there exist $n$ billiard paths from $x$ to
$y$ that have no triple intersection points except $x$ and $y$.
A ${\it triple}\ {\it intersection}\ {\it point}$ is a point at which
three or more distinct segments of the $n$ paths meet. We include
the case in which two or more of these segments come from the same
billiard path, even though we would only need to consider the case
in which all three segments come from different paths in order to
prove insecurity. The absence of triple intersection points is part
of the following definition: 
\begin{defn}
\label{def:general position} Suppose there are $n$ billiard paths
from $x$ to $y$ in the billiard table $M.$ We say that these paths
are in ${\it general\ position}$ if the following four conditions
hold:
\begin{enumerate}
\item No two paths share a vertex. 
\item No point occurs as a vertex on the same path more than once.
\item The paths have no triple intersection points except $x$ and $y$. 
\item The points $x$ and $y$ are not interior points of any of the $n$
paths.
\end{enumerate}
\end{defn}
In Theorem \ref{theorem 1} we prove that for points $x\ne y$ in
the interior of $M$ and any positive integer $n,$ there exists an
arbitrarily small $C^{r}$ perturbation $\sigma_{n}$ of $\sigma$
bounding a table $M_{n}$ such that there exist $n$ billiard paths
from $x$ to $y$ in $M_{n}$ that are in general position. Moreover,
we can do this in such a way that there exists a $C^{r}$ neighborhood
of $\sigma_{n}$ in $\Sigma_{r}$ such that if $\widetilde{\sigma}$
is in this neighborhood, then we still have $n$ billiard paths from
$x$ to $y$ in the corresponding table $\widetilde{M}$ that are
in general position (see Lemma \ref{lemma open set}). It then follows
(see Corollary \ref{cor:generic table}) that the set $\Sigma_{r,(x,y)},$
consisting of curves in $\Sigma_{r}$ that bound a region containing
$x$ and $y$ in its interior, contains a dense $G_{\delta}$ subset
of curves for which the existence of $n$ billiard paths from $x$
to $y$ that are in general position holds ${\it for\ all}$ $n.$
In Theorem \ref{thm:main result}, we apply the Kuratowski-Ulam Theorem
(the analog for Baire Category of the Fubini Theorem), to conclude
that there is a dense $G_{\delta}$ subset $A$ of $\Sigma_{r}$ such
that for each $\sigma\in A,$ the set of insecure pairs $(x,y)$ among
those pairs $(x,y)$ such that $x$ and $y$ are in the interior of
the region bounded by $\sigma$ contains a dense $G_{\delta}$ set.

To begin the proof of Theorem \ref{theorem 1}, we need a method for
generating a large collection of billiard paths from $x$ to $y.$
A ${\it polygonal\ path}$ $P_{0}P_{1}\cdots P_{k+1}$ in the plane
is defined to be the union of oriented line segments from $P_{i}$
to $P_{i+1},$ for $i=0,\dots,k,$ where $P_{0},\dots,P_{k+1}$ are
points in $\mathbb{R}^{2}.$ We fix a positive integer $k$ and points
$x$ and $y$ in a billiard table $M,$ and consider polygonal paths
$P_{0}P_{1}\dots P_{k+1}$ with $P_{0}=x,$ $P_{k+1}=y,$ and $P_{1},\dots,P_{k}\in\partial M.$
The well-known Lemma \ref{maximal length} shows that any maximal-length
polygonal path of this type is a billiard path for $M.$ We use this
lemma to obtain a billiard path $\gamma_{n+1}$ from $x$ to $y$
if we are already given billiard paths $\gamma_{1},\dots,\gamma_{n}$
from $x$ to $y$ that are in general position, and we would like
to obtain $n+1$ paths from $x$ to $y$ that are in general position.
We would like the new path $\gamma_{n+1}$ to have at least one vertex
that is not a vertex of any of $\gamma_{1},\dots,\gamma_{n}.$ However,
even if we take $k$ to be very large compared to the total number
of vertices of $\gamma_{1},\dots,\gamma_{n},$ it is possible that
the path $\gamma_{n+1}$ obtained from Lemma \ref{maximal length}
is part of a periodic billiard path, and the vertices of $\gamma_{n+1}$
are contained in the set of vertices of $\gamma_{1},\dots,\gamma_{n}.$
To avoid this problem, we make a small $C^{r}$ perturbation $\widetilde{\sigma}$
of $\sigma$ and small perturbations of $\gamma_{1},\dots,\gamma_{n}$
to obtain billiard paths $\widetilde{\gamma}_{1},\dots,\widetilde{\gamma}_{n}$
from $x$ to $y$ for the table bounded by $\widetilde{\sigma}$ such
that $\widetilde{\gamma}_{1},\dots,\widetilde{\gamma}_{n}$ are still
in general position, and, in addition, no two distinct vertices of
$\widetilde{\gamma}_{1},\dots,\widetilde{\gamma}_{n}$ are collinear
with $y.$ This implies that there is no periodic billiard path passing
through $x$ and $y$ whose vertices are a subset of the vertices
of $\widetilde{\gamma}_{1},\dots,\widetilde{\gamma}_{n}.$ Then, if
we choose $k$ sufficiently large, the new path $\gamma_{n+1}$ obtained
from Lemma \ref{maximal length} must contain at least one vertex
$V$ that is not a vertex of $\widetilde{\gamma}_{1},\dots,\widetilde{\gamma}_{n}$
(see Lemma \ref{lem:New Vertex}). By making a further perturbation
of $\widetilde{\sigma}$ if necessary, as in Lemma \ref{only once},
we may assume that $\gamma_{n+1}$ passes through $V$ only once.
Then we are free to modify the path $\gamma_{n+1}$ to $\widetilde{\gamma}_{n+1}$
by changing the initial angle at $x,$ the final angle at $y,$ and
the table near $V$ so the part of $\widetilde{\gamma}_{n+1}$ starting
at $x$ until it hits the table near $V$ and the part of $\widetilde{\gamma}_{n+1}$
from near $V$ to the point $y$ are joined to form a billiard path
for the new table (see Proposition \ref{perturb vertex and tangent}).
This procedure does not change the paths $\widetilde{\gamma}_{1},\dots,\widetilde{\gamma}_{n},$
and it can be done in such a way that $\widetilde{\gamma}_{n+1}$
does not have any vertices in common with $\widetilde{\gamma}_{1},\dots,\widetilde{\gamma}_{n}.$
Then we can make further perturbations of the table near the vertices
of $\widetilde{\gamma}_{n+1}$ using Corollaries \ref{perturb initial segment}
and \ref{parallel perturbation} to obtain $n+1$ paths in general
position.

\section{Notation and Preliminaries }

Throughout this paper we assume $r\ge2,$ and we let $\Sigma_{r}$
be the set of simple closed $C^{r}$ curves $\sigma$ in $\mathbb{R}^{2}$
such that $\sigma$ has positive curvature, that is, its acceleration
vector has a positive component in the direction of the inward-pointing
normal vector for the region enclosed by $\sigma.$ Let $M=M(\sigma)$
be the compact region in $\mathbb{R}^{2}$ bounded by some $\sigma\in\Sigma_{r}.$
For convenience, we assume $\sigma$ is parametrized by the circle
$S^{1}=\mathbb{R}/\mathbb{Z}$ and the parametrization is at constant
speed. For $s\in S^{1},$ let ${\bf T}(s)$ denote the unit tangent
vector, ${\bf T}(s)=\sigma'(s)/|\sigma'(s)|,$ and let ${\bf N}(s)$
be the unit vector perpendicular to ${\bf T}(s)$ that is inward pointing
for $M$ at $\sigma(s).$ We may assume that $\sigma$ is oriented
in the counterclockwise direction, or equivalently, the pair of vectors
${\bf T}(s),{\bf N}(s)$ has the same orientation as the standard
basis $(1,0),(0,1)$ in $\mathbb{R}^{2}.$ Note that ${\bf T}$ and
${\bf N}$ are $C^{r-1}$ functions of $s$. We let $\kappa(s)$ denote
the curvature of $\sigma$  at $\sigma(s).$ Then ${\bf T}'(s)/|\sigma'(s)|=\kappa(s){\bf N}(s),$
with $\kappa(s)>0.$ 

A billiard path $\gamma(t)$, defined for $t\in\mathbb{R}$, represents
the position at time $t$ of a point mass moving within $M$ at unit
speed with elastic collisions at $\partial M.$ More precisely, there
is a partition $\cdots<c_{-2}<c_{-1}<c_{0}<c_{1}<c_{2}<\cdots$ of
$\mathbb{R}$ with $c_{0}\le0<c_{1}$ such that for each $i\in\mathbb{Z}$,
$\gamma|[c_{i-1},c_{i}]$ is a line segment in $M$ parametrized at
unit speed, $\gamma(c_{i})\in\partial M,$ and 
\[
\gamma'_{-}(c_{i})=(\cos\alpha_{i}){\bf T}-(\sin\alpha_{i}){\bf N}
\]

\[
\gamma'_{+}(c_{i})=(\cos\alpha_{i}){\bf T}+(\sin\alpha_{i}){\bf N},
\]
where ${\bf T}={\bf T}(s_{i})$ and ${\bf N}={\bf N}(s_{i}),$ with
$s_{i}\in S^{1}$ chosen so that $\gamma(c_{i})=\sigma(s_{i});$ $\alpha_{i}\in(0,\pi)$
is given by $\alpha_{i}=\varangle(\gamma'_{-}(c_{i}),{\bf T});$ and
$\gamma'_{-}$ $[\gamma'_{+}]$ denotes the derivative of $\gamma$
from the left {[}right{]}. It follows that 
\begin{equation}
\gamma'_{+}(c_{i})=\gamma'_{-}(c_{i})-2\text{\text{Proj}}_{{\bf N}}{\bf \gamma'_{-}}(c_{i}),\label{eq:reflected vector}
\end{equation}
where Proj$_{w}v$ denotes the orthogonal projection of $v$ onto
span($w$). We also consider billiard paths from a point $x\in{\rm Int}(M)$
to a point $y\in{\rm Int}(M).$ Such a path is a billiard path $\gamma(t)$
defined as above, except the domain of $\gamma$ is restricted to
a compact interval $[a,b]$, $\gamma(a)=x,$ and $\gamma(b)=y$. The
points $\gamma(c_{i})\in\partial M$ with $c_{i}\in(a,b)$ are called
the ${\it vertices}$ of the path. We will refer to the segment of
the path from $x$ to the first vertex as the ${\it initial\ segment}$,
and the segment of the path from the last vertex to $y$ as the ${\it final\ segment.}$
It is also possible for the billiard path from $x$ to $y$  to simply
be the segment $xy,$ in which case there are no vertices, and the
segment $xy$ is both the initial and the final segment. A segment
of the path that joins two consecutive vertices is called a ${\it chord}$
of the path. 

For $p\in{\rm Int}(M)$, let $T_{p}^{1}M$ denote the set of all unit
vectors based at $p$; for $p\in\partial M$, let $T_{p}^{1}M$ denote
the set of unit vectors based at $p$ that are inward-pointing for
$M$ (not including vectors tangent to $\partial M).$ Vectors in
$T_{p}^{1}M$ will be written in the form $(p,v),$ where $v\in S^{1},$
or simply as $v$ if the basepoint is understood. The phase space
of the billiard flow is $\cup_{p\in M}T_{p}^{1}M.$ For $v\in T_{p}^{1}M,$
the billiard flow $\Psi^{t}((p,v)):=(\gamma(t),\gamma_{+}'(t))$ where
$\gamma$ is the billiard path such that $\gamma(0)=p$ and $\gamma'_{+}(0)=v.$
The standard section of the billiard flow is the map $\Phi:\Gamma\to\Gamma,$
where $\Gamma:=\cup_{p\in\partial M}T_{p}^{1}M,$ and for $(p,v)\in\Gamma,$
$\Phi(v):=(\gamma(c_{1}),\gamma'_{+}(c_{1})),$ where $\gamma$ is
the billiard path with $\gamma(0)=p$ and $\gamma'_{+}(0)=v$, and
$t=c_{1}$ is the first time after $t=0$ such that $\gamma(t)\in\partial M.$ 

Let $x,y\in{\rm Int(}M)$. Denote by $L_{m}$ a polygonal path $P_{0}P_{1}\ldots P_{m}P_{m+1}$
of maximal length, subject to the conditions $P_{0}=x,$ $P_{m+1}=y,$
and $P_{i}\in\partial M$ for all $i\in\{1,2,\dots,m\}$. Since $(P_{1},\ldots P_{m})\in\partial M\times\cdots\times\partial M$,
which is a compact set, such an $L_{m}$ exists. Note that all consecutive
vertices of $L_{m}$ must be distinct in order for $L_{m}$ to have
maximal length: If $P_{i}=P_{i+1}$, then by the triangle inequality,
the path $L_{m}'$ formed by moving $P_{i+1}$ slightly away from
$P_{i}$ on $\sigma$ has length longer than $L_{m}$. The following
well-known lemma will be useful in obtaining a collection of distinct
billiard paths from $x$ to $y.$ 
\begin{lem}
\label{maximal length} For $m=1,2,\dots,$ and given points $x,y\in{\rm Int}(M)$,
a maximal-length polygonal path $L_{m}=P_{0}P_{1}\dots P_{m}P_{m+1},$
subject to the conditions $P_{0}=x$, $P_{m+1}=y,$ and $P_{i}\in\partial M$
for all $i\in\{1,2,\dots,m\},$ is a billiard trajectory.\end{lem}
\begin{proof}
It suffices to prove that for any points $A$ and $B$ in $M$, a
polygonal path $AZB,$ with $Z\in\partial M,$ of maximal length must
be a billiard path. This follows from the Lagrange multiplier principle
applied to the function $f(Z)=|AZ|+|ZB|.$ (See pp. 12-13 of \cite{Tab05}.)
\end{proof}
In order to study $D\Phi$ geometrically it is convenient to introduce
$C^{k}$ families of oriented lines in $\mathbb{R}^{2},$ for $k\ge1.$
In our application, we usually take $k=r-1.$
\begin{defn}
\label{linefamily} Let $I\subset\mathbb{R}$ be an compact interval
of positive length, and for each $u\in I,$ suppose that $\ell(u)$
is an oriented line in $\mathbb{R}^{2}.$ For $k\ge1,$ we say that
$\ell(u),u\in I,$  is a $C^{k}$ \emph{family of oriented lines }if
there exist $C^{k}$ functions ${\bf \xi}:I\to\mathbb{R}^{2}$ and
$v:I\to S^{1}$ such that for each $u\in I,$ $\ell(u)=\{\xi(u)+tv(u):t\in\mathbb{R}\},$
and the orientation on $\ell(u)$ is given by $v(u).$ The point $\xi(u)$
and the vector $v(u)$ are the \emph{base point }and the \emph{direction
vector,} respectively, of the line $\ell(u).$ Such a family $\ell(u)$
is said to be \emph{non-degenerate} if the following condition holds:
For each $u\in I,$ if $v'(u)=0,$ then $\xi'(u)$ is not a scalar
multiple of $v(u).$ (It follows that $\xi'(u)$ and $v'(u)$ are
not both $0.)$ In particular, the non-degeneracy condition prevents
the family of oriented lines $\ell(u)$ from consisting of just a
single line. 
\end{defn}
If $\ell(u)=\{\xi(u)+tv(u):t\in\mathbb{R}\},$ $u\in I,$ is a $C^{k}$
family of oriented lines, then we may consider parameter translations
along each $\ell(u)$ by letting $\widetilde{\xi}(u)=\xi(u)+a(u)v(u),$
where $a:I\to\mathbb{R}$ is $C^{k}.$ Note that the definition of
non-degeneracy is satisfied by $(\xi,v)$ if and only if it is satisfied
by $(\xi,v)$ replaced by $(\widetilde{\xi},v).$ Therefore the definition
of non-degeneracy is independent of parameter translations. We may
also reparametrize the family $\ell(u),$ $u\in I,$ as $\ell(\beta(\widetilde{u})),$
$\widetilde{u}\in\widetilde{I},$ where $\widetilde{I}$ is a compact
interval of positive length and $\beta:\widetilde{I}\to I$ is a $C^{k}$
diffeomorphism. The definition of non-degeneracy is also independent
of such a reparametrization.
\begin{defn}
\label{focusing} If $\ell(u),$ $u\in I,$ is a $C^{1}$ family of
lines parametrized by $\ell(u,t)=\xi(u)+tv(u),$ $u\in I,$ $t\in\mathbf{R},$
then the \emph{$local$ envelope} of $\ell(u)$ is defined to be $f(u)=-\left<\xi'(u),v'(u)\right>/\left<v'(u),v'(u)\right>$
wherever $v'(u)\ne0.$ The point $F:=\ell(u_{0},f(u_{0}))$ is said
to be the \emph{focusing point }(in linear approximation) for the
family $\ell(u)$ at $u=u_{0}.$ If $v'(u_{0})=0,$ we say the focusing
point is at infinity for $u=u_{0}.$ (See, e.g., Section 2 of \cite{Woj86}.)
\end{defn}
A straightforward computation shows that the focusing point for the
family $\ell(u)$ at $u=u_{0}$ does not change under parameter translations
along each of the $\ell(u).$ Likewise, if the family $\ell(u),$
$u\in I,$ is reparametrized as $\ell(\beta(\widetilde{u})),$ $\widetilde{u}\in\widetilde{I},$
where $\beta$ is as above, then the focusing point for the family
$\ell(u)$ at $u=u_{0}\in I$ is the same as the focusing point for
the family $\ell(\beta(\widetilde{u}))$ at $\widetilde{u}=\beta^{-1}(u_{0})\in\widetilde{I}.$ 
\begin{lem}
\label{invertible}Suppose $\ell(u),$ $u\in I,$ is a non-degenerate
$C^{1}$ family of lines parametrized by $\ell(u,t)=\xi(u)+tv(u),$
$t\in\mathbb{R}.$ For $(u_{0},t_{0})\in I\times\mathbb{R},$ $D\ell(u_{0},t_{0})$
is invertible if and only if $\ell(u_{0},t_{0})$ is not a focusing
point at $u=u_{0}.$ \end{lem}
\begin{proof}
The derivative of $\ell(u,t)$ is given by
\[
D\ell(u_{0},t_{0})=[\begin{array}{cc}
\xi'(u_{0})+t_{0}v'(u_{0}) & v(u_{0})\end{array}].
\]

Case 1. Suppose $v'(u_{0})=0.$ Then $\ell(u_{0},t_{0})$ is not a
focusing point at $u=u_{0},$ and the columns of $D\ell(u_{0},t_{0})$
are linearly independent, since $v(u_{0})$ is a unit vector and the
non-degeneracy condition implies that $\xi'(u_{0})$ is not a scalar
multiple of $v(u_{0}).$

Case 2. Suppose $v'(u_{0})\ne0.$ Since $\left<v'(u),v(u)\right>\equiv0,$
the columns of $D\ell(u_{0},t_{0})$ are linearly dependent if and
only if the orthogonal projection of $\xi'(u_{0})+t_{0}v'(u_{0})$
onto $v'(u_{0})$ is 0. Thus the columns of $D\ell(u_{0},t_{0})$
are linearly dependent if and only if

\[
\dfrac{\left<\xi'(u_{0})+t_{0}v'(u_{0}),v'(u_{0})\right>}{\left<v'(u_{0}),v'(u_{0})\right>}=0,
\]
which is equivalent to $\ell(u_{0},t_{0})$ being a focusing point
at $u=u_{0}.$ 

\end{proof}
\begin{defn}
\label{reflected lines} Let $\ell(u),$ $u\in I,$ be a $C^{k}$
family of oriented lines, where $1\le k\le r-1.$ Assume that the
line $\ell(u)$ intersects ${\rm Int}(M)$ for each $u\in I.$ Suppose
the lines are parametrized by $\ell(u,t)=\xi(u)+tv(u),$ $t\in\mathbf{R},$
where $\xi:I\to M$ and $v:I\to S^{1}$ are $C^{k}$ functions and
$(\xi(u),v(u))\in T_{\xi(u)}^{1}M,$ for all $u,$ which implies that
$v(u)$ is inward pointing for $M$ at $\xi(u)$ if $\xi(u)\in\partial M.$
The ${\it reflected}\ family\ of\ oriented\ lines,$ $\ell_{1}(u),$
obtained from $\ell(u)$ is defined as follows: For each $u\in I,$
there are exactly two values of $t$, say $t_{a}=t_{a}(u)$ and $t_{b}=t_{b}(u),$
where $t_{a}<t_{b},$ such that $\xi(u)+t_{a}v(u)$ and $\xi(u)+t_{b}v(u)$
are on $\partial M.$ Let $\xi_{1}(u)=\xi(u)+t_{b}(u)v(u)$ and let
$v_{1}(u)=v(u)-2\text{Proj}_{N(\xi_{1}(u))}v(u).$ By (\ref{eq:reflected vector}),
we see that a billiard trajectory along the line $\ell(u)$ (going
in the same direction as the orientation of $\ell(u))$ is reflected
at $\xi_{1}(u)$ on $\partial M,$ so that it continues along the
line $\ell_{1}(u)$ after reflection. The reflected family, $\ell_{1}(u),$
$u\in I,$ is the family of oriented lines that can be parametrized
by $\ell_{1}(u,t)=\xi_{1}(u)+tv_{1}(u),$ $t\in\mathbb{R}.$\end{defn}
\begin{lem}
\label{reflection lemma} Let $M$ be a billiard table in $\mathbb{R}^{2}$
bounded by a curve $\sigma\in\Sigma_{r},$ where $r\ge2.$ Suppose
$\ell(u),$ $u\in I,$ is a $C^{k}$ family of oriented lines, where
$1\le k\le r-1,$ and $\ell(u)$ intersects ${\rm Int}(M)$ for each
$u\in I.$ Then the reflected family, $\ell_{1}(u),$ of oriented
lines obtained from $\ell(u)$ is also $C^{k}.$ Moreover, if the
family $\ell(u)$ is non-degenerate, then so is $\ell_{1}(u).$ \end{lem}
\begin{proof}
Let $\xi,\xi_{1},v,v_{1},t_{a},t_{b}$ be as in Definition \ref{reflected lines}.
From the inverse function theorem and the compactness of $S^{1}$
it follows that there exists $\delta>0$ such that the function $\varphi:S^{1}\times(-\delta,\delta)\to\mathbb{R}^{2}$
given by $\varphi(s,w)=\sigma(s)+w{\bf N}(s)$ is injective and defines
a $C^{r-1}$ coordinate system on a neighborhood $\mathcal{N}_{\delta}:=\{p\in\mathbb{R}^{2}:\text{dist}(p,\partial M)<\delta\}.$
This is a special case of the tubular neighborhood theorem (see, e.g.,
\cite{BurGid}). The $w$ coordinate gives the signed distance of
a point in $\mathcal{N}_{\delta}$ to $\partial M,$ with the sign
chosen so that $w<0$ outside $M$ and $w>0$ in ${\rm Int}(M).$
Since $\nabla w={\bf N}(\sigma(s))$ is a $C^{r-1}$ function of $s,$
$w$ is a $C^{r}$ function of the usual $(x,y)$ coordinates on $\mathcal{N}_{\delta}.$
For $(u,t)\in I\times\mathbb{R}$ such that $\ell(u,t)\in\partial M,$
we have $\partial(w\circ\ell)/\partial t=\left\langle {\bf \nabla w},v(u)\right\rangle =\left\langle {\bf N}(\ell(u,t)),v(u)\right\rangle \ne0.$
Therefore, by the implicit function theorem, the function $t_{b}(u)$,
which satisfies the equation $w\circ\ell(u,t_{b}(u))=0,$ is $C^{k}.$
Hence the function $\xi_{1}:I\to\partial M$ defined by $\xi_{1}(u):=\xi(u)+t_{b}(u)v(u)$
is $C^{k}.$ By (\ref{eq:reflected vector}), we have $v_{1}(u):=v(u)-2\text{Proj}_{{\bf N}(\xi_{1}(u))}v(u)=v(u)-2\left\langle {\bf N}(\xi_{1}(u)),v(u)\right\rangle {\bf N}(\xi_{1}(u)),$
which is $C^{k}.$ Therefore, $\ell_{1}(u)$ is $C^{k}.$

Now assume, in addition, that $\ell(u)$ is non-degenerate. Note that
$\ell(u)$ can be reparametrized by $\ell(u,t)=\xi_{1}(u)+tv(u).$
Thus $\xi_{1}(u)$ and $v(u)$ satisfy the condition for non-degeneracy.
We need to prove that $\xi_{1}(u)$ and $v_{1}(u)$ also satisfy this
condition. We have $v_{1}'(u)=v'(u)-2\left\langle \xi_{1}'(u){\bf N}'(\xi_{1}(u)),v(u)\right\rangle {\bf N}(\xi_{1}(u))-2\left\langle {\bf N}(\xi_{1}(u)),v'(u)\right\rangle {\bf N}(\xi_{1}(u))-2\left\langle {\bf N}(\xi_{1}(u)),v(u)\right\rangle ({\bf N}'(\xi_{1}(u))(\xi_{1}'(u)).$
Since $\xi_{1}(u)\in\partial M,$ we consider the following two cases
for each $u_{0}\in I$:
\begin{enumerate}
\item Suppose $\xi_{1}'(u_{0})$ is a non-zero scalar multiple of ${\bf {\bf T}=T}(\xi_{1}(u_{0})).$
Then $\xi_{1}'(u_{0})$ is not a multiple of $v_{1}(u_{0})$, since
the line $\ell(u_{0})$, as well as its reflection $\ell_{1}(u_{0}),$
intersects Int$(M)$, which implies that $v_{1}(u_{0})$ and ${\bf T}$
are transversal.
\item Suppose $\xi_{1}'(u_{0})=0.$ Then by the non-degeneracy of $\ell(u),\ \ \ \ \ \ $
$v'(u_{0})\ne0.$ The formula for $v_{1}'(u_{0})$ simplifies to
\[
\begin{array}{ccc}
v_{1}'(u_{0}) & = & v'(u_{0})-2\left\langle {\bf N}(\xi_{1}(u_{0})),v'(u_{0})\right\rangle {\bf N}(\xi_{1}(u_{0}))\\
 & = & v'(u_{0})-2\text{Proj}_{{\bf {\bf N}}(\xi_{1}(u_{0}))}v'(u_{0})\ \ \ \ \ \ \ \ \ \ \ \ \ \ \\
 & = & \text{Proj}_{{\bf T}(\xi_{1}(u_{0}))}v'(u_{0})-\text{Proj}_{{\bf N}(\xi_{1}(u_{0}))}v'(u_{0}).
\end{array}
\]
 Thus in the $({\bf T},{\bf N})$ coordinates $v_{1}'(u_{0})$ is
obtained from $v'(u_{0})$ by changing the sign of the ${\bf N}$
coordinate. Therefore $v_{1}'(u_{0})\ne0.$ 
\end{enumerate}

Therefore the family $\ell_{1}(u)$ is non-degenerate.

\end{proof}
We now give explicit formulations of the standard $C^{k}$ distance
between functions, families of lines, and curves. For later use, we
also include a definition of distance between oriented polygonal paths.
\begin{defn}
If $K$ is a compact subset of $\mathbb{R}^{n}$ and $f,\widetilde{f}:K\rightarrow\mathbb{R}^{m}$
are $C^{k}$ maps, then the $C^{k}$ distance between $f$ and $\widetilde{f}$
is defined to be 
\begin{equation}
{\rm dist}_{C^{k}(K,\mathbb{R}^{m})}(f,\widetilde{f}):=\sup_{j\in\{0,1,\dots,k\}}\ \sup_{s\in K}|D^{(j)}f(s)-D^{(j)}\widetilde{f}(s)|.\label{eq:C^rDistance}
\end{equation}
If $\ell(u,t)=\xi(u)+tv(u)$ and $\widetilde{\ell}(u,t)=\widetilde{\xi}(u)+t\widetilde{v}(u),$
$u\in I$, $t\in\mathbb{R},$ are two parametrized families of lines
where $\xi,\widetilde{\xi},v,\widetilde{v}$ are $C^{k}$ functions,
then the $C^{k}$ distance between $\ell$ and $\widetilde{\ell}$
depends on the parametrizations and is defined to be
\[
d_{k}(\ell,\widetilde{\ell}):=\max\ \{{\rm dist}_{C^{k}(I,\mathbb{R}^{2})}(\xi,\widetilde{\xi}),{\rm dist}_{C^{k}(I,\mathbb{R}^{2})}(v,\widetilde{v})\}.
\]
If $\sigma$ and $\widetilde{\sigma}$ are simple closed $C^{k}$
curves in the plane, we define the $C^{k}$ distance geometrically
as follows. Let $f_{p},\widetilde{f}:S^{1}\rightarrow\mathbb{R}^{2}$
be constant speed parametrizations of $\sigma$ and $\widetilde{\sigma}$,
respectively, with $f_{p}(0)=p\in\sigma.$ We regard $S^{1}=\mathbb{R}/\mathbb{\mathbb{Z}}$
to be $[0,1]$ with the endpoints identified and apply formula (\ref{eq:C^rDistance})
to $f_{p}$ and $\widetilde{f}$ with $K=[0,1]:$ 

\[
d_{k}(\sigma,\widetilde{\sigma}):=\inf_{p\in\sigma}d_{k}(f_{p},\widetilde{f}).
\]
For $0\le k\le r,$ we refer to the topology on $\Sigma_{r}$ induced
by the metric $d_{k}$ as the $C^{k}$ topology on $\Sigma_{r.}$ 
\end{defn}
If $\gamma=P_{0}P_{1}\cdots P_{m+1}$ and $\widehat{\gamma}=\widehat{P}_{0}\widehat{P}_{1}\cdots\widehat{P}_{m+1}$
are (oriented) polygonal paths in $\mathbb{R}^{2},$ as defined in
Section 2, then the distance between $\gamma$ and $\widehat{\gamma}$
is 
\[
d(\gamma,\widehat{\gamma}):=\max\{{\rm dist}(P_{i},\widehat{P}_{i}):i=0,\dots,m+1\}.
\]
We assume that $P_{i},P_{i+1},P_{i+2}$ are noncollinear and $\widehat{P}_{i},\widehat{P}_{i+1},\widehat{P}_{i+2}$
are noncollinear for $i=0,\dots,m-1,$ so that the points $P_{0},P_{1},\dots,P_{m+1}$
and the points $\widehat{P}_{0},\widehat{P}_{1},\dots,\widehat{P}_{m+1}$
are uniquely determined by $\gamma$ and $\widehat{\gamma},$ respectively. 

The following lemma is a ``perturbed'' version of the first part
of Lemma \ref{reflection lemma}.
\begin{lem}
\label{perturbed}Let $r\ge2,$ $\sigma\in\Sigma_{r},$ and $M=M(\sigma)$
be the billiard table bounded by $\sigma.$ Suppose $\ell(u,t)=\xi(u)+tv(u),$
$u\in I,$ $t\in\mathbb{R}$ is a $C^{k}$ parametrized family of
lines with $1\le k\le r-1$ such that for each $u\in I$ the line
$\ell(u):=\ell(u,\cdot)$ intersects $\text{Int}(M).$ Then for every
$\epsilon>0$ there exists $\delta>0$ such that if $\widetilde{\sigma}\in\Sigma_{r}$
and $\widetilde{\ell}(u,t)=\widetilde{\xi}(u)+t\widetilde{v}(u),$
$u\in I,$ $t\in\mathbb{R}$ is a $C^{k}$ parametrized family of
lines, then the family $\ell_{1}(u,t)=\xi_{1}(u)+tv_{1}(u),$ obtained
by reflecting $\ell(u,t)$ from $\sigma,$ and the family $\widetilde{\ell}_{1}(u,t)=\widetilde{\xi}_{1}(u)+t\widetilde{v}_{1}(u)$,
obtained by reflecting $\widetilde{\ell}(u,t)$ from $\widetilde{\sigma},$
satisfy $d_{k}(\ell_{1},\widetilde{\ell}_{1})<\epsilon$ whenever
$d_{r}(\sigma,\widetilde{\sigma})<\delta$ and $d_{k}(\ell,\widetilde{\ell})<\delta.$
Here we require $\xi_{1}$ and $\widetilde{\xi}_{1}$ to be chosen,
as in the proof of Lemma \ref{reflection lemma}, so that $\xi_{1}(u)$
and $\widetilde{\xi}_{1}(u)$ lie on the images of $\sigma$ and $\widetilde{\sigma},$
respectively.\end{lem}
\begin{proof}
We use the same notation as in the proof of Lemma \ref{reflection lemma},
with all objects obtained from $\widetilde{\sigma}$ and $\widetilde{\ell}$
in an analogous way to objects obtained from $\sigma$ and $\ell$
being given a tilde over them. Let $\ell(u,t)=\xi(u)+tv(u)$ and $\sigma$
be as in the statement of the lemma. By taking $\delta$ sufficiently
small, we may assume that each $\widetilde{\ell}(u)$ intersects $\text{Int}(\widetilde{M}).$
By following the proof of Lemma \ref{reflection lemma}, we see that
for small $\delta,$ the signed distances $\widetilde{w}$ and $w$
from $\widetilde{\sigma}$ and $\sigma$ are both defined in a neighborhood
of $\sigma,$ $\nabla\widetilde{w}=\widetilde{N}\circ\widetilde{\sigma}$
is $C^{r-1}$ close to $\nabla w=N\circ\sigma,$ and $\widetilde{w}$
is $C^{r}$ close to $w.$ It then follows that the functions $t_{b}(u)$
and $\widetilde{t}_{b}(u)$ obtained from the implicit function theorem
such that $(\widetilde{w}\circ\widetilde{\ell})(u,\widetilde{t}_{b}(u))=0=(w\circ\ell)(u,t_{b}(u))$
are $C^{r}$ close. Using the procedure of Lemma \ref{reflection lemma}
to find $\xi_{1},\widetilde{\xi}_{1},v_{1},\widetilde{v}_{1},$ we
see that for sufficiently small $\delta,$ $d_{k}(\ell,\widetilde{\ell}_{1})<\epsilon.$
\end{proof}

\section{Nonconjugacy Condition for Billiard Paths}

In this section we discuss the conjugacy condition, which has proved
to be useful in the theory of billiards. (See \cite{Bis03} for further
information on the adaptation of this condition, as well as other
related notions from differential geometry, to the study of billiards.)
In Corollary \ref{cor: Neigborhood filling}, we present a consequence
of nonconjugacy in the form that we need in Section 5. 

Let the billiard table $M$ be as before, with boundary curve $\sigma\in\Sigma_{r},$
where $r\ge2.$ 

Suppose $\ell(u),$ $u\in I,$ is a $C^{k}$ family of oriented lines,
where $1\le k\le r-1,$ such that each $\ell(u)$ intersects ${\rm Int(}M),$
and $\widetilde{\ell}(u),$ $u\in I,$ is the reflected family of
lines as in Definition \ref{reflected lines}. Suppose the parametrizations
of $\ell(u,t)=\xi(u)+tv(u)$ and $\widetilde{\ell}(u,t)=\widetilde{\xi}(u)+t\widetilde{v}(u)$
are chosen so that for some $u_{0}\in I,$ $\xi(u_{0})=\widetilde{\xi}(u_{0})$,
and $\xi(u_{0})$ lies on $\sigma.$ Let $f,\widetilde{f}\in\mathbb{R}$
be such that $F=\ell(u_{0},f)$ and $\widetilde{F}=\widetilde{\ell}(u_{0},\widetilde{f})$
are the focusing points of the families $\ell$ and $\widetilde{\ell},$
respectively, at $u=u_{0}.$ Let $\kappa$ be the curvature of $\sigma$
at $\xi(u_{0}),$ and let $\alpha$ be the angle that $v(u_{0})$
makes with the tangent vector ${\bf T}$ to $\sigma$ at $\xi(u_{0}).$
Then according to the mirror equation (see, e.g., Theorem 5.28 in
\cite{Tab05} or Lemma 1 in Section 2 of \cite{Woj86}), 
\begin{equation}
-\frac{1}{f}+\frac{1}{\widetilde{f}}=\frac{2\kappa}{\sin\alpha}.\label{eq:Mirror}
\end{equation}
 Note that $\sin\alpha\ne0$ for a billiard table whose boundary has
positive curvature. In case $f$ or $\widetilde{f}$ is equal to $\infty,$
we interpret $\frac{1}{\infty}$ as $0.$ If either of $f$ or $\widetilde{f}$
is $0,$ then the other one is also $0$.
\begin{defn}
\label{conjugate} Let $p$ and $q$ be points in $M$, and let $\tau$
be a billiard path with initial point $p$, final point $q,$ and
$m$ reflections between the initial and final points. Let $\ell_{0}$
be the oriented line determined by the initial segment of $\tau.$
Then $p$ and $q$ are said to be \textit{conjugate along} $\tau$
\textit{in} $M$ if the following holds: There exists a non-degenerate
$C^{1}$ family of lines $\ell_{0}(u),$ $u\in I,$ with $\ell_{0}(u_{0})=\ell_{0}$
for some $u_{0}\in I$, such that the families $\ell_{0}(u),\ell_{1}(u),\dots,\ell_{m}(u),$
where $\ell_{i}(u)$ is obtained by reflecting $\ell_{i-1}(u)$ from
$\partial M$ (as in Definition \ref{reflected lines}) for $i=1,\dots,m,$
are such that $\ell_{0}(u)$ and $\ell_{m}(u)$ focus at $p$ and
$q,$ respectively, at $u=u_{0}.$ 
\end{defn}
According to Corollary \ref{cor: Neigborhood filling} below, if $\tau$
is a billiard path from $p$ to $q$ for a given table such that $p$
and $q$ are not conjugate along $\tau$, and we make a small $C^{2}$
perturbation of the table and a small perturbation $q_{1}$ of $q,$
then there is a unique billiard path from $p$ to $q_{1}$ in the
new table that is close to $\tau$. For this we need the following
perturbed version of the inverse function theorem. Its proof is the
same as that of the usual inverse function theorem (see, e.g., \cite{Rud76}). 
\begin{lem}
\label{lemma inverse function theorem} Let $K=\overline{B_{\epsilon_{0}}(x_{0})}\subset\mathbb{R}^{2},$
where $B_{\epsilon_{0}}(x_{0})$ denotes the open ball of radius $\epsilon_{0}$
about some point $x_{0}\in\mathbb{R}^{2}.$ Suppose $f:K\rightarrow\mathbb{R}^{2}$
is $C^{1}$ and $Df(x_{0})$ is invertible. Let $y_{0}=f(x_{0}).$
Then there exist $\eta,\delta>0$ and $0<\epsilon<\epsilon_{0},$
such that for any $C^{1}$ function $g:K\rightarrow\mathbb{\mathbb{R}}^{2}$
with $\textrm{dist}_{C^{1}(K,\mathbb{R}^{2})}(f,g)<\eta,$ we have
$g$ one-to-one on $B_{\epsilon}(x_{0})$ and $B_{\delta}(y_{0})\subset g(B_{\epsilon}(x_{0}))$. \end{lem}
\begin{cor}
\label{cor: Neigborhood filling} Let $p$ and $q$ be points in $M=M(\sigma),$
where $\sigma\in\Sigma_{2}.$ Let $\tau$ be a billiard path in $M$
from $p$ to $q$ that makes $m\ge0$ reflections between $p$ and
$q$. Assume that $p$ and $q$ are not conjugate along $\tau$ in
$M.$ Let $\epsilon_{1}>0.$ Then there exist $\eta,\delta>0$ and
$0<\epsilon<\epsilon_{1}$ such that the following implication holds:
If
\begin{enumerate}
\item $\widetilde{\sigma}\in\Sigma_{2},$
\item $d_{2}(\sigma,\widetilde{\sigma})<\eta,$
\item $p,q_{1}\in\widetilde{M},$ where $\widetilde{M}=M(\widetilde{\sigma}),$
\item ${\rm dist}(q,q_{1})<\delta$,
\end{enumerate}
then there is a \emph{unique} billiard path for $\widetilde{M}$ starting
at $p$ and making angle less than $\epsilon$ with $\tau$ at $p$
that passes through $q_{1}$ after reflection number $m$ and before
(or at) reflection number $m+1$ from $\partial\widetilde{M}.$ \end{cor}
\begin{proof}
If $m=0,$ the result clearly holds, so assume $m>0.$ Let $\ell_{0},\ell_{1},...,\ell_{m}$
be the oriented lines determined by the segments of $\tau$. Let $\ell_{0}(u)$
be the non-degenerate $C^{1}$ family of lines that pass through $p$
and whose direction vectors $v(u)$ are such that the signed angle
from the initial tangent vector of $\tau$ to $v(u)$ is $u.$ For
$i=1,\dots,m$, let $\ell_{i}(u)$ be the family of lines obtained
by reflecting $\ell_{i-1}(u)$ from $\sigma$. Suppose $\widetilde{\sigma}$
is a $C^{2}$ perturbation of $\sigma$. Let $\widetilde{\ell}_{0}(u)=\ell_{0}(u)$,
and for $i=1,\dots,m$, let $\widetilde{\ell}_{i}(u)$ be the family
of lines obtained by reflecting $\widetilde{\ell}_{i-1}(u)$ from
$\widetilde{\sigma}$. Let $t_{0}>0$ be such that $\ell_{m}(0,t_{0})=q$.
Let $\epsilon_{0},\eta_{1}>0$ be sufficiently small so that $(p,v(u))\in T^{1}M\cap T^{1}\widetilde{M}$
if $|u|<\epsilon_{0},$ $d_{2}(\sigma,\widetilde{\sigma})<\eta_{1},$
and $p\in\widetilde{M}.$ (This restriction of $u$ is only needed
if $p\in\partial M.$) 

Let $\epsilon_{1}>0$. Let $f(u,t)=\ell_{m}(u,t)$ and let $x_{0}=(0,t_{0}).$
We may assume that $\epsilon_{1}<\min(t_{0}/2,\epsilon_{0}).$ Let
$K:=\overline{B_{\epsilon_{1}}(x_{0})}\subset(-\epsilon_{0},\epsilon_{0})\times(t_{0}/2,3t_{0}/2),$
which is contained in the domain of $f.$ Since $p$ is a focusing
point for $\ell_{0}(u)$ at $u=0,$ it follows from the nonconjugacy
assumption and Lemma \ref{invertible} that $Df(x_{0})$ is invertible.
Let $g(u,t)=\widetilde{\ell}_{m}(u,t)$. By Lemma \ref{lemma inverse function theorem}
there exist $\epsilon_{2}\in(0,\epsilon_{1})$ and $\alpha>0$ such
that ${\rm dist}_{C^{1}(K,\mathbb{R}^{2})}(f,g)<\alpha$ implies that
$\widetilde{\ell}_{m}$ is one-to-one on $B_{2\epsilon_{2}}(x_{0}).$
By repeated application of Lemma \ref{perturbed}, there exists $\eta_{2}\in(0,\eta_{1})$
such that ${\rm dist}_{C^{1}(K,\mathbb{R}^{2})}(f,g)<\alpha$ whenever
$d_{2}(\sigma,\widetilde{\sigma})<\eta_{2}.$ 

Note that $d(\ell_{m}(0,t),q)\ge\epsilon_{2}$ for $t\notin(t_{0}-\epsilon_{2},t_{0}+\epsilon_{2})$.
Thus, there exist $\epsilon\in(0,\epsilon_{2})$ and $\eta_{3}\in(0,\eta_{2})$
such that for $|u|<\epsilon$ and $d_{2}(\sigma,\widetilde{\sigma})<\eta_{3},$
we have $d(\widetilde{\ell}_{m}(u,t),q)\ge\epsilon_{2}/2$ whenever
$t\notin(t_{0}-\epsilon_{2},t_{0}+\epsilon_{2}).$ By Lemma \ref{perturbed}
and Lemma \ref{lemma inverse function theorem}, there exist $\eta\in(0,\eta_{3})$
and $\delta\in(0,\epsilon_{2}/2)$ such that if $d_{2}(\sigma,\widetilde{\sigma})<\eta$
and $q_{1}\in\widetilde{M}$ with ${\rm dist}(q,q_{1})<\delta,$ then
there is a point $(u_{1},t_{1})\in B_{\epsilon}(x_{0})$ with $\widetilde{\ell}_{m}(u_{1},t_{1})=q_{1}.$
This proves the existence of a billiard path for $\widetilde{M}$
with the required properties. 

The point $(u_{1},t_{1})$ is the unique point in $B_{2\epsilon_{2}}(x_{0})$
with $\widetilde{\ell}_{m}(u_{1},t_{1})=q_{1},$ because $\widetilde{\ell}_{m}$
is one-to-one on $B_{2\epsilon_{2}}(x_{0}).$ Now suppose $|u|<\epsilon$,
$t\in\mathbb{R}$, and $\widetilde{\ell}_{m}(u,t)=q_{1}.$ Since $d(\widetilde{\ell}_{m}(u,t),q)\ge\epsilon_{2}/2>\delta$
for $t\notin(t_{0}-\epsilon_{2},t_{0}+\epsilon_{2})$, we must have
$t\in(t_{0}-\epsilon_{2},t_{0}+\epsilon_{2}).$ But then $(u,t)\in B_{2\epsilon_{2}}(x_{0}),$
and therefore $(u,t)=(u_{1},t_{1})$, which proves the uniqueness
assertion. 
\end{proof}

\section{Construction of $n$ paths without triple intersections}

As before, we let $M$ be the compact region bounded by a curve $\sigma\in\Sigma_{r}$,
where $r\ge2$. Suppose $x$ and $y$ are distinct points in ${\rm Int}(M).$
Recall that for $n$ billiard paths from $x$ to $y$ in the billiard
table $M,$ we say that the paths are in general position if the following
four conditions hold: 
\begin{enumerate}
\item [(GP1)]No two paths share a vertex. 
\item [(GP2)]No point occurs as a vertex on the same path more than once.
\item [(GP3)]The paths have no triple intersection points except $x$ and
$y$. 
\item [(GP4)]The points $x$ and $y$ are not interior points of any of
the $n$ paths.
\end{enumerate}
We say that the paths satisfy the ${\it non}$-$collinearity\ condition$
if the following holds: 
\begin{enumerate}
\item [(NC)]For any distinct vertices $P$ and $Q$ of the $n$ paths,
the points $P$, $Q,$ and $y$ are not collinear. 
\end{enumerate}
Note that condition (GP4) implies none of the $n$ paths is perpendicular
to $\sigma$ at any vertex. Moreover, if there exist $n$ paths from
$x$ to $y$ in general position for each positive integer $n$, then
$(x,y)$ is insecure for the billiard table $M.$ Condition (NC) is
not essential for our applications, but it is useful to establish
this condition in addition to conditions (GP1)--(GP4) in our proof
of Theorem \ref{theorem 1}.

We consider the billiard table $M,$ as well as billiard tables obtained
by perturbing the boundary of $M.$ In this section we prove the following.
\begin{thm}
\label{theorem 1}Suppose $M,\sigma,r,x,y$ are as above. Let $n$
be a positive integer and let $\epsilon>0.$ Then there exists a curve
$\sigma_{n}\in\Sigma_{r}$ with $d_{r}(\sigma,\sigma_{n})<\epsilon$
that bounds a region $M_{n}$ still containing $x$ and $y$ in its
interior such that for the billiard system on $M_{n}$ there are $n$
billiard paths from $x$ to $y$ that are in general position.
\end{thm}
We proceed toward a proof of Theorem \ref{theorem 1}, as outlined
in Section 2. We first observe that the property of paths being in
general position is preserved under small perturbations of the paths.

Let $x$ and $y$ be distinct points in $\mathbb{R}^{2}.$ If $\gamma=P_{0}P_{1}\cdots P_{m+1}$
is a polygonal path from $x$ to $y,$ i.e., $P_{0}=x$ and $P_{m+1}=y,$
then we refer to $P_{1},\dots,P_{m}$ as the vertices of $\gamma.$
(In our application $\gamma$ will be a billiard path, and $P_{1},\dots,P_{m}$
will also be vertices in the sense of billiard paths, i.e., they will
lie on the boundary of the billiard table.) As for billiard paths,
we say that polygonal paths $\gamma_{1},\dots,\gamma_{n}$ from $x$
to $y$ are in general position if conditions (GP1)--(GP4) hold, and
we say they satisfy the non-collinearity condition if (NC) holds.
The following lemma is easy to verify, and we omit the proof. 
\begin{lem}
\label{lem:preservation of general position}Suppose $\gamma_{1},\dots,\gamma_{n}$
are polygonal paths from $x$ to $y$ that are in general position.
Then there exists $\epsilon>0$ such that for any polygonal paths
$\widehat{\gamma}_{1},\dots,\widehat{\gamma}_{n}$ from $\widehat{x}$
to $\widehat{y}$ such that $\widehat{\gamma}_{i}$ has the same number
of vertices as $\gamma_{i}$ and $d(\gamma_{i},\widehat{\gamma}_{i})<\epsilon$
for $i=1,\dots,n,$ it follows that $\widehat{\gamma}_{1},\dots,\widehat{\gamma}_{n}$
are in general position. This result remains true if we replace ``are
in general position'' by ``satisfy the non-collinearity condition''
in both the hypothesis and the conclusion. 
\end{lem}
In the next lemma, we show how the non-collinearity condition is used
in the proof of Theorem \ref{theorem 1}. 
\begin{lem}
\label{lem:New Vertex} Let $M,\sigma,r,x,y$ be as above. Suppose
that $\gamma_{1},\dots,\gamma_{n}$ are billiard paths from $x$ to
$y$ in $M$ that satisfy the non-collinearity condition (NC). Then
there exists a billiard path $\gamma$ from $x$ to $y$ in $M$ which
has at least one vertex $V$ that is not a vertex of any of $\gamma_{1},\dots,\gamma_{n}.$\end{lem}
\begin{proof}
Let $\mathcal{P}=\left\{ P_{1},P_{2},\ldots P_{k}\right\} $ be the
vertices of $\gamma_{1},\dots,\gamma_{n},$ and let $N$ be an integer
greater than $k^{2}-k+1$. By Lemma \ref{maximal length}, there exists
a billiard path $\mbox{\ensuremath{\gamma}}$ from $x$ to $y$ of
the form $xQ_{1}Q_{2}\cdots Q_{N}y$ with vertices $Q_{1},Q_{2},\dots,Q_{N}.$
If each $Q_{i},$ $i=1,2,\dots,N,$ were in $\mathcal{P}$ then at
least two of the ordered pairs $(Q_{j},Q_{j+1}),$ $j\in\{1,2,\dots N-1\}$
would be the same. But this would imply that the path is contained
in a periodic billiard path through $x$ and $y,$ which is impossible
due to condition (NC). Therefore the set $\{Q_{1},Q_{2},\dots,Q_{N}\}$
must include at least one point not in $\mathcal{P}.$ 
\end{proof}
Proposition \ref{change curvature} and Corollary \ref{nonconjugacy }
below allow us to avoid unwanted conjugacies between any given pairs
of points $(p_{i},q_{i})$ along billiard paths $\gamma_{i},$ for
$i=1,\dots,n,$ by making a small $C^{r}$ change to the table, while
$\gamma_{1},\dots,\gamma_{n}$ remain billiard paths for the new table. 
\begin{prop}
\label{change curvature} Suppose $r\ge2,$ $P$ is a point on a curve
$\sigma\in\Sigma_{r},$ and $\sigma$ has curvature $\kappa$ at $P.$
Given $\epsilon>0,$ there exists $\delta>0$ such that for any $\widetilde{\kappa}\in(\kappa-\delta,\kappa+\delta)$
there exists a curve $\widetilde{\sigma}\in\Sigma_{r}$ such that
$\widetilde{\sigma}$ passes through $P$, $\widetilde{\sigma}$ is
tangent to $\sigma$ at $P$, the curvature of $\widetilde{\sigma}$
at $P$ is $\widetilde{\kappa}$, $\widetilde{\sigma}$ agrees with
$\sigma$ outside an $\epsilon$-neighborhood of $P$, and $d_{r}(\widetilde{\sigma},\sigma)<\epsilon$.\end{prop}
\begin{proof}
Choose $\nu>0$ and a Cartesian coordinate system centered at $P$
such that the box $B=\{(x,y):|x|\leq\nu,|y|\leq\nu\}$ is contained
in the $\epsilon$-neighborhood of $P$, and within this box $\sigma$
is given by the graph of a function $f:[-\nu,\nu]\rightarrow[-\nu/2,\nu/2]$
with $f(0)=f'(0)=0$ and $f''(0)>0$. 

Let $\Psi:\mathbb{R}\rightarrow\mathbb{R}$ be a $C^{\infty}$ function
such that $\Psi(0)=\Psi'(0)=0$, $\Psi''(0)\neq0$, and $\Psi(x)=0$
for $x\notin(-\nu,\nu)$. For $|c|$ sufficiently small, the function
$g(x)=f(x)+c\Psi(x)$ maps $[-\nu,\nu]$ to $[-\nu,\nu]$. Let $\widetilde{\sigma}$
be the curve obtained by replacing the graph of $f$ by the graph
of $g$ and keeping $\widetilde{\sigma}=\sigma$ outside $B$. Then
$\widetilde{\sigma}$ passes through $P$, $\widetilde{\sigma}$ is
tangent to $\sigma$ at $P$, and the curvature of $\widetilde{\sigma}$
at $P$ is $\kappa+c\Psi''(0)$. There exists $\eta>0$ such that
for $|c|<\eta,$ we have $\widetilde{\sigma}\in\Sigma_{r}$ and $d_{r}(\widetilde{\sigma},\sigma)<\epsilon$.
Then the result holds with $\delta=\eta|\Psi''(0)|.$ \end{proof}
\begin{cor}
\label{nonconjugacy } Let $M,\sigma,r$ be as above. Suppose the
polygonal path $\gamma=V_{0}V_{1}\cdots V_{m}V_{m+1}$ is a billiard
path for $M,$ where $p=V_{0}\in M,$ $q=V_{m+1}\in M,$ and $V_{1},\dots,V_{m}\in\partial M.$
(The points $p,q$ are allowed to be on $\partial M.)$ Given $\epsilon>0,$
there exists $\widetilde{\sigma}\in\Sigma_{r}$ bounding a table $\widetilde{M}$
such that $d_{r}(\sigma,\widetilde{\sigma})<\epsilon$ and the following
conditions are satisfied:
\begin{enumerate}
\item $\gamma$ is still a billiard path in $\widetilde{M}$; 
\item $\sigma$ and $\widetilde{\sigma}$ agree outside the $\epsilon$-neighborhoods
of $V_{1},\dots,V_{m};$
\item $p$ is not conjugate to $q$ along $\gamma$ in $\widetilde{M}.$
\end{enumerate}

Moreover, for any sufficiently small $C^{2}$ perturbation $\widetilde{\widetilde{\sigma}}$
of $\widetilde{\sigma}$ that agrees with $\widetilde{\sigma}$ to
first order at each of $V_{1},\dots,V_{m},$ $p$ is not conjugate
to $q$ along $\gamma$ in the table $\widetilde{\widetilde{M}}$
bounded by $\widetilde{\widetilde{\sigma}}$. 

\end{cor}
\begin{proof}
Assume that $p$ and $q$ are conjugate along $\gamma.$ (Otherwise,
we may take $\widetilde{\sigma}=\sigma.)$ Let $\kappa_{i}$ be the
curvature of $\sigma$ at $V_{i}$ for $i=1,\dots,m,$ and let $\rho_{i}={\rm dist}(V_{i},V_{i+1})$
for $i=0,\dots,m.$ Let $\ell_{0},\dots,\ell_{m}$ be the oriented
lines along the segments of $\gamma,$ that is, $\ell_{i}=V_{i}V_{i+1},$
for $i=0,\dots,m.$ Parametrize $\ell_{i}$ by $V_{i+1}+tv_{i},$
where $v_{i}$ is the unit vector in the direction of $\ell_{i}.$
Let $\ell_{0}(u),$ with $u\in I=[-\delta,\delta]$ for some $\delta>0,$
be the family of lines through $p$ with $\ell_{0}(0)=\ell_{0},$
and $\ell_{0}(u)$ making signed angle $u$ with $\ell_{0}.$ Suppose
$\widetilde{\sigma}=\widetilde{\sigma}(z)$ is the boundary of a billiard
table $\widetilde{M}$ such that $\sigma$ and $\widetilde{\sigma}$
agree to first order at $V_{1},\dots,V_{m},$ and $\widetilde{\sigma}(z)$
has curvature $\kappa_{i}z$ at $V_{i},$ for $i=1,\dots,m.$ Let
$\widetilde{\ell}_{0}(u)=\ell_{0}(u),$ and for $i=1,\dots,m,$ let
$\widetilde{\ell}_{i+1}(u)$ be the family of lines obtained by reflecting
$\widetilde{\ell}_{i}(u)$ from $\partial\widetilde{M}.$ Let $f_{i}=f_{i}(z)\in\mathbb{R}\cup\{\infty\}$
be such that $V_{i+1}+f_{i}v_{i}$ is the focusing point for $\widetilde{\ell}_{i}(u),$
$u\in I,$ at $u=0.$ Then $f_{0}=-\rho_{0},$ and for $i=0,\dots,m-1,$
it follows from the mirror equation (\ref{eq:Mirror}) that
\end{proof}

\subjclass[2000]{
\begin{equation}
f_{i+1}=\frac{1}{\frac{1}{f_{i}}+\frac{2\kappa_{i}z}{\sin\alpha_{i}}}-\rho_{i+1},\label{eq:FocusingRecursion}
\end{equation}
where $\alpha_{i}$ is the angle that $\gamma$ makes with the tangent
line to $\sigma$ at the $(i+1)$st reflection. The term $-\rho_{i+1}$
occurs in equation (\ref{eq:FocusingRecursion}) because the origin
(time $t=0)$ of $\ell_{i+1}$ is taken at $V_{i+2}$ while the origin
of $\ell_{i}$ is at $V_{i+1}.$ By repeated application of (\ref{eq:FocusingRecursion})
we see that $f_{m}=f_{m}(z)$ is a linear fractional transformation
of $z,$ that is, it is of the form $(az+b)/(cz+d)$. Note that $p$
and $q$ are conjugate along $\gamma$ within $\widetilde{M}$ if
and only if $f_{m}(z)$=0. We want to show that we may choose $z$
arbitrarily close to $1$ so that $f_{m}(z)\ne0.$ But if we take
$z=0,$ we obtain $f_{m}(0)=-\rho_{0}-\rho_{1}-\dots-\rho_{m}\ne0.$
Therefore $f_{m}(z)$ is not identically $0.$ Thus there exists $z$
arbitrarily close to $1$ so that $f_{m}(z)\ne0.$ Given $\epsilon>0,$
it follows from Proposition \ref{change curvature} that for $z$
sufficiently close to $1,$ we can find $\widetilde{\sigma}\in\Sigma_{r}$
such that $d_{r}(\sigma,\widetilde{\sigma})<\epsilon,$ $\widetilde{\sigma}$
agrees with $\sigma$ to first order at $V_{1},\dots,V_{m},$ $\sigma$
and $\widetilde{\sigma}$ agree outside the $\epsilon$-neighborhoods
of $V_{1},\dots,V_{m},$ and the curvature of $\widetilde{\sigma}$
at $V_{i}$ is $\kappa_{i}z$ for $i=1,\dots,m.$ Thus, for $z$ close
to 1, but not equal to 1, conditions (1)--(3) are satisfied for $\widetilde{\sigma}.$ }
\begin{proof}
It follows from (\ref{eq:FocusingRecursion}) that non-conjugacy of
$p$ and $q$ along $\gamma$ within $\widetilde{M}$ is preserved
when $\widetilde{\sigma}$ is replaced by a small $C^{2}$ perturbation
$\widetilde{\widetilde{\sigma}}$ that agrees with $\widetilde{\sigma}$
to first order at $V_{1},\dots,V_{m}.$\end{proof}
\begin{rem}
In the proof of Corollary \ref{nonconjugacy }, if one of the vertices
$V_{i_{0}}$ occurs only once on the list $V_{1},\dots,V_{m},$ then
$\widetilde{\sigma}$ could be obtained by changing the curvature
just at $V_{i_{0}}.$ The argument that we used takes care of the
case in which each $V_{i_{0}}$ may occur multiple times, and we have
to make sure that the effect on $f_{m}$ of changing the curvature
at the vertices is not canceled out during multiple passages of $\gamma$
through these vertices. 
\end{rem}
The following proposition provides a perturbation technique that will
be used repeatedly in the proof of Theorem \ref{theorem 1}.
\begin{prop}
\label{perturb vertex and tangent}Let $r\ge2$ and $\sigma\in\Sigma_{r}.$
Suppose $P$ is a point on $\sigma$ and $\ell$ is the tangent line
to $\sigma$ at $P$. Given $\epsilon>0$, there exists $\delta>0$
such that if $\widetilde{P}\in\mathbb{R}^{2}$ is such that $d(P,\widetilde{P})<\delta$
and $\widetilde{\ell}$ is a line through $\widetilde{P}$ that is
either parallel to $\ell$ or makes angle less than $\delta$ with
$\ell$, then there exists $\widetilde{\sigma}\in\Sigma_{r}$ such
that $\widetilde{\sigma}$ agrees with $\sigma$ outside an $\epsilon$-neighborhood
of $P$, $\widetilde{\sigma}$ passes through $\widetilde{P}$ with
tangent line $\widetilde{\ell}$ at $\widetilde{P},$ and $d_{r}(\sigma,\widetilde{\sigma})<\epsilon$.\end{prop}
\begin{proof}
Let $\epsilon>0.$ Choose $\nu>0$ and a Cartesian coordinate system
centered at $P$ such that the box $B=\{(x,y):|x|\leq\nu,|y|\leq\nu\}$
is contained in the $\epsilon$-neighborhood of $P$, and within this
box $\sigma$ is given by the graph of a function $f:[-\nu,\nu]\rightarrow[-\nu/2,\nu/2]$
with $f(0)=f'(0)=0$ and $f''(0)>0$. Suppose that in this coordinate
system $\widetilde{P}=(a,b)$ and $\widetilde{\ell}$ has slope $m$. 

Let $\Psi_{1}:\mathbb{R}\rightarrow\mathbb{R}$ be a $C^{\infty}$
function such that $\Psi_{1}(x)=0$ for all $x\notin(-\nu/2,\nu/2)$,
$\Psi_{1}(0)=1$, and $\Psi_{1}'(0)=0$. Let $\psi_{2}:\mathbb{R}\rightarrow\mathbb{R}$
be a $C^{\infty}$ function such that $\psi_{2}(0)=1$, $\int_{-\nu/2}^{0}\psi_{2}(x)~dx=\int_{0}^{\nu/2}\psi_{2}(x)~dx=0$,
and $\psi_{2}(x)=0$ for $x\notin(-\nu/2,\nu/2)$. Let $\Psi_{2}:\mathbb{R}\rightarrow\mathbb{R}$
be defined by $\Psi_{2}(x)=\int_{-\infty}^{x}\psi_{2}(u)~du$. Then
$\Psi_{2}(0)=0$, $\Psi_{2}'(0)=1$, and $\Psi_{2}(x)=0$ for $x\notin(-\nu/2,\nu/2)$. 

Define a function $\Psi:[-\nu,\nu]\rightarrow\mathbb{R}$ by
\[
\Psi(x)=(b-f(a))\Psi_{1}(x-a)+(m-f'(a))\Psi_{2}(x-a).
\]
If $|a|<\nu/2$, then $\Psi$ vanishes in a neighborhood of the boundary
of $[-\nu,\nu]$. For any $\eta>0$, since $f$ is $C^{1}$ there
exists $\delta=\delta(\eta)>0$ such that $\max(|a|,|b|,|f(a)|,|f'(a)|,|m|)<\eta$
whenever $d(P,\widetilde{P})<\delta$ and $\ell$ and $\widetilde{\ell}$
are either parallel or meet at an angle less than $\delta$. 

For $\eta$ sufficiently small, the function $g(x)=f(x)+\Psi(x)$
maps $[-\nu,\nu]$ to $[-\nu,\nu]$. Let $\widetilde{\sigma}$ be
the curve obtained by replacing the graph of $f$ by the graph of
$g$ and keeping $\widetilde{\sigma}=\sigma$ outside $B$. Since
$g(a)=b$ and $g'(a)=m$, we see that $\widetilde{\sigma}$ passes
through $\widetilde{P}$ and has tangent line $\widetilde{\ell}$
at $\widetilde{P}$. Moreover, if $\eta$ is sufficiently small, then
$\widetilde{\sigma}\in\Sigma_{r}$ and $d_{r}(\sigma,\widetilde{\sigma})<\epsilon.$ \end{proof}
\begin{cor}
\label{perturb initial segment}Consider a billiard path $APA_{1}$
in the table M with boundary $\sigma\in\Sigma_{r}$, for $r\ge2,$
where $A,A_{1}\in{\rm Int}(M)$ and $P\in\partial M.$ For any $\epsilon>0$,
if $\widetilde{P}$ is a point on either of the segments $AP$ or
$PA_{1}$ and is sufficiently close to $P$, there exists $\widetilde{\sigma}\in\Sigma_{r}$
with $d_{r}(\sigma,\widetilde{\sigma})<\epsilon$ such that $A\widetilde{P}A_{1}$
is a billiard path for the table with boundary $\widetilde{\sigma}$.
(See Figure 1.)\end{cor}
\begin{proof}
Let $\ell$ be the tangent line to $\sigma$ at $P.$ Then $\ell$
is perpendicular to the angle bisector of $\varangle(APA_{1}).$ If
$\widetilde{P}$ on $AP$ or $PA_{1}$ is sufficiently close to $P,$
then the line $\widetilde{\ell}$ that is perpendicular to the angle
bisector of $\varangle(A\widetilde{P}A_{1})$ makes a small angle
with $\ell.$ Then we apply Proposition \ref{perturb vertex and tangent}.\end{proof}
\begin{cor}
\label{parallel perturbation} Let $APQB$ be a billiard path in the
table $M$ with boundary $\sigma\in\Sigma_{r}$, for $r\ge2,$ where
$A,B\in{\rm Int}(M)$ and $P,Q\in\partial M.$ For any $\epsilon>0$,
if $\widetilde{P}$ is a point on the line $AP$ and is sufficiently
close to $P$, then there exists $\widetilde{\sigma}\in\Sigma_{r}$
bounding a table $\widetilde{M}$ such that $d_{r}(\sigma,\widetilde{\sigma})<\epsilon,$
$\widetilde{\sigma}$ agrees with $\sigma$ outside $\epsilon$-neighborhoods
of $P$ and $Q,$ and $A\widetilde{P}\widetilde{Q}B$ is a billiard
path for $\widetilde{M}$, where $\widetilde{Q}$ is the point where
the line through $\widetilde{P}$ parallel to $PQ$ intersects the
line $BQ.$ (See Figure 2.)\end{cor}
\begin{proof}
In order for $A\widetilde{P}\widetilde{Q}B$ to be a billiard path,
the tangent lines to $\tilde{\sigma}$ at $\widetilde{P}$ and $\widetilde{Q}$
need to be perpendicular to the angle bisectors of $A\widetilde{P}\widetilde{Q}$
and $\widetilde{P}\widetilde{Q}B$, respectively. If $\widetilde{P}$
is sufficiently close to $P$, then $\widetilde{Q}$ is sufficiently
close to $Q$ that the angle bisectors of $A\widetilde{P}\widetilde{Q}$
and $\widetilde{P}\widetilde{Q}B$ are close to the angle bisectors
of $APQ$ and $PQB$, respectively. Thus $\widetilde{\sigma}$ can
be obtained by applying Proposition \ref{perturb vertex and tangent}
twice.\end{proof}
\begin{lem}
\label{only once}Let $r\ge2$, $\sigma\in\Sigma_{r},$ and $M=M(\sigma)$
be as above. Let $x$ and $y$ be distinct points in ${\rm Int}(M).$
Suppose $\gamma_{1},\ldots,\gamma_{n+1}$ are billiard paths in $M$
from $x$ to $y$ such that $\gamma_{1},\ldots,\gamma_{n}$ are in
general position, and $\gamma_{n+1}$ has at least one vertex $V$
that is not a vertex of any of $\gamma_{1},\ldots,\gamma_{n}.$ Then
given $\epsilon>0$ there exists $\widetilde{\sigma}\in\Sigma_{r}$
with $d_{r}(\sigma,\widetilde{\sigma})<\epsilon$ such that for the
table $\widetilde{M}$ bounded by $\widetilde{\sigma},$ we have $x,y\in{\rm Int}(\widetilde{M}),$
and there are billiard paths $\widetilde{\gamma}_{1},\ldots,\widetilde{\gamma}_{n},\widetilde{\gamma}_{n+1}$
from $x$ to $y$ satisfying the following conditions:\end{lem}
\begin{enumerate}
\item $d(\gamma_{i},\widetilde{\gamma}_{i})<\epsilon$ for $i=1,\dots,n+1,$
\item $\widetilde{\gamma}_{1},\dots,\widetilde{\gamma}_{n}$ are in general
position,
\item $\widetilde{\gamma}_{n+1}$ has a vertex $\widetilde{V}$ with ${\rm dist}(V,\widetilde{V})<\epsilon,$
where $\widetilde{V}$ is not a vertex of any of $\widetilde{\gamma}_{1},\dots,\widetilde{\gamma}_{n},$ 
\item $\widetilde{\gamma}_{n+1}$ passes through $\widetilde{V}$ only once,
and
\item $\widetilde{\gamma}_{n+1}$ is not perpendicular to $\widetilde{\sigma}$
at any of its vertices. \end{enumerate}
\begin{proof}
Let $\epsilon>0.$ By Corollary \ref{nonconjugacy }, we may assume
that $x$ and $y$ are not conjugate along any of $\gamma_{1},\dots,\gamma_{n},$
and $V$ is not conjugate to itself along any path contained in $\gamma_{n+1}.$
We will assume $\epsilon$ is sufficiently small so that $d(\gamma_{i},\widehat{\gamma}_{i})<\epsilon$
for $i=1,\dots,n$ implies that $V$ is not a vertex of any of $\widehat{\gamma}_{1},\dots,\widehat{\gamma}_{n};$
and, in addition, by Lemma \ref{lem:preservation of general position},
$d(\gamma_{i},\widehat{\gamma}_{i})<\epsilon$ for $i=1,\dots,n$
implies that $\widehat{\gamma}_{1},\dots,\widehat{\gamma}_{n}$ are
in general position. Thus condition (2) will follow from condition
(1) for $i=1,\dots,n.$ Also, condition (3) will follow from condition
(1) for $\epsilon$ sufficiently small$.$ By Corollary \ref{cor: Neigborhood filling}
and Lemma \ref{perturbed}, there exists $\delta>0$ such that for
any $i\in\{1,2,\dots,n\},$ any $\widehat{\sigma}\in\Sigma_{r}$ with
$d_{2}(\sigma,\widehat{\sigma})<\delta,$ and any $y_{1}\in\widehat{M}$
with dist$(y,y_{1})<\delta,$ where $\widehat{M}$ is the region bounded
by $\widehat{\sigma},$ there exist billiard paths $\widehat{\gamma}_{i}$
for $\widehat{M}$ from $x$ to $y_{1}$ with $d(\gamma_{i},\widehat{\gamma}_{i})<\epsilon.$
We may assume that $\delta$ is sufficiently small so that whenever
dist$(y,y_{1})<\delta$ and $d_{r}(\widehat{\sigma},\sigma)<\epsilon/2,$
the curve $\widetilde{\sigma}$ obtained from $\widehat{\sigma}$
by applying the composition of a homothety (uniform expansion or contraction)
centered at $x$ and a rotation of the plane centered at $x$ such
that $y_{1}$ gets mapped to $y$ satisfies $d_{r}(\widetilde{\sigma},\sigma)<\epsilon.$

Given $0<\eta<\epsilon/2,$ there exists $\alpha_{0}>0$ such that
for $|\alpha|<\alpha_{0},$ Proposition \ref{perturb vertex and tangent}
provides a $C^{r}$ perturbation $\sigma_{\alpha}$ of $\sigma$ such
that: $d_{r}(\sigma,\sigma_{\alpha})<\eta;$ $\sigma_{\alpha}$ agrees
with $\sigma$ outside a small neighborhood of $V$ that is chosen
sufficiently small that it contains no vertex of $\gamma_{n+1}$ other
than $V$ and it contains no vertices of $\gamma_{1},\dots,\gamma_{n};$
$V$ lies on $\sigma_{\alpha};$ the tangent line at $V$ for $\sigma_{\alpha}$
is obtained by rotating the tangent line at $V$ for $\sigma$ by
angle $\alpha$ in the counterclockwise direction; and the table $M_{\alpha}$
bounded by $\alpha$ contains $x$ and $y$ in its interior. Let $\gamma_{n+1,\alpha}$
be the billiard path for $M_{\alpha}$ obtained by starting with the
initial part of $\gamma_{n+1}$ from $x$ to the first occurrence
of $V$ along $\gamma_{n+1}$ and then continuing from $V$ with the
same number of reflections as the number of reflections $\gamma_{n+1}$
has after the first occurrence of $V$. Note that the signed angle
from the tangent vector of $\gamma_{n+1}$ to the tangent vector of
$\gamma_{n+1,\alpha}$ at the time that these paths leave $V$ for
the first time is $2\alpha.$ We end $\gamma_{n+1,\alpha}$ after
the last reflection at a point $y_{1}$ in $M_{\alpha}$ that is as
close as possible to $y.$ Suppose $\gamma_{n+1}$ passes through
$V$ the first time at reflection number $k_{1}$, the second time
at reflection number $k_{2},$ etc., up to the $s$th time at reflection
number $k_{s},$ where we count the reflections starting with the
first reflection after $x.$ By Corollary \ref{cor: Neigborhood filling},
we may assume that $\alpha_{0}$ and $\eta$ are sufficiently small
so that there are unique angles $\alpha_{1},\dots,\alpha_{s}$ with
$|\alpha_{j}|\le\alpha_{0}$ such that $\gamma_{n+1,\alpha}$ passes
through $V$ at reflection number $k_{j}$ for $j=1,\dots,s.$ We
may also assume that $\alpha_{0}$ and $\eta$ are sufficiently small
so that for $|\alpha|<\alpha_{0},$ we have $d(\gamma_{n+1,\alpha},\gamma_{n+1})<\delta,$
and, in addition, $d(\gamma_{n+1,\alpha},\gamma_{n+1})$ is sufficiently
small so that $\gamma_{n+1,\alpha}$ does not pass through $V$ at
the $i$th reflection for $i\notin\{k_{1},\dots,k_{s}\}.$ For the
rest of the proof, we fix a choice of $\alpha$ with $|\alpha|<\alpha_{0}$
such that $\alpha\notin\{\alpha_{1},\dots,\alpha_{s}\}$ and $\gamma_{n+1}$
is not perpendicular to the tangent line to $\sigma_{\alpha}$ at
$V$ at the first time that $\gamma_{n+1}$ visits $V.$ Let $\widehat{\gamma}_{n+1}=\gamma_{n+1,\alpha}$,
$\widehat{M}=M_{\alpha},$ and $\widehat{\sigma}=\sigma_{\alpha}.$
Then $d_{r}(\widehat{\sigma},\sigma)<\eta<\epsilon/2.$ Since $d(\widehat{\gamma}_{n+1},\gamma_{n+1})<\delta,$
the final endpoint $y_{1}$ of $\widehat{\gamma}_{n+1}$ satisfies
${\rm dist}(y_{1},y)<\delta.$ 

Let $\widehat{\gamma}_{1},\dots,\widehat{\gamma}_{n}$ be paths in
general position from $x$ to $y_{1}$ for the billiard table bounded
by $\widehat{\sigma}$ that we obtain from the first paragraph of
the proof. Then $\widehat{\gamma}_{1},\dots,\widehat{\gamma}_{n}$
do not pass through $V,$ $\widehat{\gamma}_{n+1}$ passes through
$V$ exactly once, and $\widehat{\gamma}_{n+1}$ is not perpendicular
to $\widehat{\sigma}$ at $V.$ We may assume that $\widehat{\gamma}_{n+1}$
is not perpendicular to $\widehat{\sigma}$ at any vertex other than
$V,$ because if it were, we could avoid this by replacing $\widehat{\gamma}_{n+1}$
by the shortest path within $\widehat{\gamma}_{n+1}$ that starts
at $x$, ends at $y_{1},$ and contains $V.$ We have now achieved
everything required, except that the paths $\widehat{\gamma}_{1},\dots,\widehat{\gamma}_{n+1}$
end at $y_{1}$ instead of $y.$ By the choice of $\delta$ in the
first paragraph, if we apply the composition of a homothety and a
rotation, both centered at $x,$ so that $y_{1}$ is sent to $y,$
then $\widehat{\sigma}$ is mapped to a curve $\widetilde{\sigma}$
with $d_{r}(\widetilde{\sigma},\sigma)<\epsilon,$ the billiard paths
$\widehat{\gamma}_{1},\dots,\widehat{\gamma}_{n+1}$ for $\widehat{M}$
get mapped to billiard paths $\widetilde{\gamma}_{1},\dots,\widetilde{\gamma}_{n+1}$
for the billiard table $\widetilde{M}$ bounded by $\widetilde{\sigma},$
and $V$ is sent to a vertex $\widetilde{V}$ of $\widetilde{\gamma}_{n+1}.$
Then $\widetilde{\gamma}_{1},\dots,\widetilde{\gamma}_{n+1}$ and
$\widetilde{V}$ satisfy conditions (1)--(5). 
\end{proof}
We now complete the proof of the main result of this section.
\begin{proof}
[Proof of Theorem \ref{theorem 1}.] If $n=1,$ we can take $\gamma_{1}$
to be the segment $xy.$ Then $\gamma_{1}$ has no vertices, and conditions
(GP1)--(GP4) and (NC) are clearly satisfied with $\sigma_{1}=\sigma.$ 

Let $\epsilon>0,$ and let $(\epsilon_{k})_{k=1}^{\infty}$ be a sequence
of positive numbers with sum less than $\epsilon.$ Assume $n$ is
a positive integer, $\sigma_{n}\in\Sigma_{r},$ $d_{r}(\sigma,\sigma_{n})<\epsilon_{1}+\epsilon_{2}+\cdots+\epsilon_{n-1},$
and we have $n$ billiard paths $\gamma_{1},\ldots,\gamma_{n}$ in
the table $M_{n}$ bounded by $\sigma_{n}$ that satisfy conditions
(GP1)--(GP4) and (NC). For the inductive argument, we will show that
there exists $\sigma_{n+1}\in\Sigma_{r}$ with $d_{r}(\sigma_{n},\sigma_{n+1})<\epsilon_{n}$
such that there are $n+1$ billiard paths in the table $M_{n+1}$
bounded by $\sigma_{n+1}$ that satisfy conditions (GP1)--(GP4) and
(NC). Each of the finitely many perturbations that we will make in
obtaining $M_{n+1}$ from $M_{n}$ can be made arbitrarily small in
the $d_{r}$ metric.

By Lemma \ref{lem:New Vertex} and Lemma \ref{only once}, there exists
a new table $\widetilde{M}_{n}$ on which there are paths $\widetilde{\gamma}_{1},\dots,\widetilde{\gamma}_{n},\gamma_{n+1}$
such that $\gamma_{n+1}$ has a vertex $V$ that is not a vertex of
any of $\widetilde{\gamma}_{1},\dots,\widetilde{\gamma}_{n}$ and
such that $\gamma_{n+1}$ passes through $V$ only once and is not
perpendicular to the boundary of $\widetilde{M}_{n}$ at any vertex.
In applying Lemma \ref{only once}, $\widetilde{\gamma}_{1},\ldots,\widetilde{\gamma}_{n}$
can be chosen as close to $\gamma_{1},\ldots,\gamma_{n}$ as we like,
and therefore by Lemma \ref{lem:preservation of general position}
we may assume that $\widetilde{\gamma}_{1},\ldots,\widetilde{\gamma}_{n}$
still satisfy (GP1)--(GP4) and (NC).

Next, we describe a perturbation of the table in a small neighborhood
of $V$ (that does not contain any vertices of $\widetilde{\gamma}_{1},\dots,\widetilde{\gamma}_{n})$
and a perturbation $\widetilde{\gamma}_{n+1}$ of $\gamma_{n+1}$
such that $\widetilde{\gamma}_{n+1}$ is a billiard path on the perturbed
table and such that $\widetilde{\gamma}_{1},\ldots,\widetilde{\gamma}_{n},\widetilde{\gamma}_{n+1}$
satisfy condition (GP1) on the perturbed table. We may assume the
neighborhood of $V$ where this perturbation takes place is small
enough that $\widetilde{\gamma}_{1},\ldots,\widetilde{\gamma}_{n}$
are unchanged. To obtain (GP1), it suffices to ensure that $\widetilde{\gamma}_{n+1}$
does not pass through any vertices of $\widetilde{\gamma}_{1},\ldots,\widetilde{\gamma}_{n}$.
Suppose there exist vertices of $\widetilde{\gamma}_{1},\ldots,\widetilde{\gamma}_{n}$
that occur in $\gamma_{n+1}.$ (If no such vertices exist, we are
finished with condition (GP1).) We may label these vertices so that
$P_{1},\dots,P_{k}$ occur along $\gamma_{n+1}$ between $x$ and
$V$, and $P_{k+1},\dots,P_{\ell}$ occur between $V$ and $y$. Now
by Corollary \ref{nonconjugacy } we may assume that $\widetilde{M}_{n}$
is constructed so that $x$ and $P_{i}$ are not conjugate along $\gamma_{n+1}$
for $i=1,\dots,k$, and $P_{i}$ and $y$ are not conjugate along
$\gamma_{n+1}$ for $i=k+1,\dots,\ell$. Therefore, by Corollary \ref{cor: Neigborhood filling}
there exists a billiard path $\tau_{x}$ starting at $x$ with initial
direction making a small but nonzero angle with that of $\gamma_{n+1}$
and ending at the boundary of $\widetilde{M}_{n}$ near $V$ that
avoids $P_{1},\dots,P_{k}$. Similarly, there exists a billiard path
$\tau_{y}$ starting at $y$ with initial direction making a small
but nonzero angle with that of $-\gamma_{n+1}$ and ending at the
boundary of $\widetilde{M}_{n}$ near $V$ that avoids $P_{k+1},\dots,P_{\ell}$.
Since $\gamma_{n+1}$ is not perpendicular to the boundary of $\widetilde{M}_{n}$
at $V$, the lines containing the segments of $\gamma_{n+1}$ with
endpoint $V$ intersect transversely at $V.$ Therefore the final
segments of the paths $\tau_{x}$ and $\tau_{y}$, possibly slightly
extended beyond $\widetilde{M}_{n},$ intersect at some point $\widetilde{V}$
near $V$, and we end these final segments at $\widetilde{V.}$ The
angle bisector of the final segments of $\tau_{x}$ and $\tau_{y}$
is arbitrarily close to the angle bisector of the segments of $\gamma_{n+1}$
that have $V$ as an endpoint. By Proposition \ref{perturb vertex and tangent},
we can now make a small $C^{r}$ perturbation of the table in a small
neighborhood of $V$ so that the paths $\tau_{x}$ and $-\tau_{y}$
can be joined at $\widetilde{V}$ to form a billiard path $\widetilde{\gamma}_{n+1}$
for the new table, which we call $\widetilde{\widetilde{M}}_{n},$
and $\widetilde{\gamma}_{n+1}$ avoids the vertices of $\widetilde{\gamma}_{1},\ldots,\widetilde{\gamma}_{n}.$
This perturbation does not affect $\widetilde{\gamma}_{1},\ldots,\widetilde{\gamma}_{n}.$
Therefore $\widetilde{\gamma}_{1},\dots,,\widetilde{\gamma}_{n+1}$
satisfy condition (GP1) for $\widetilde{\widetilde{M}}_{n}.$ Moreover,
$\widetilde{\gamma}_{n+1}$ passes through $\widetilde{V}$ only once
and $\widetilde{\gamma}_{n+1}$ is not perpendicular to the boundary
of $\widetilde{\widetilde{M}}_{n}$ at any vertex. 

Now we want to satisfy condition (GP2). Suppose $\widetilde{\gamma}_{n+1}$
passes through a vertex $P\ne\widetilde{V}$ more than once. By condition
(GP1), $P$ is not a vertex of any of $\widetilde{\gamma}_{1},\dots,\widetilde{\gamma}_{n}.$
Then we can replace $V$ by $P$ in Lemma \ref{only once} and obtain
a small $C^{r}$ perturbation of the boundary of the table, and small
perturbations of $\widetilde{\gamma}_{1},\dots,\widetilde{\gamma}_{n+1}$
maintaining condition (GP1) for the $n+1$ paths, maintaining conditions
(GP2)--(GP4) and (NC) for the first $n$ paths, and having the $(n+1)$st
path go through a vertex $\widetilde{P}$ near $P$ only once. This
procedure can be repeated until the $(n+1)$st path goes through each
vertex only once. Note that the procedure ends after finitely many
steps because the number of reflections made by the $(n+1)$st path
is unchanged by these perturbations. Thus we obtain condition (GP2)
for paths $\widehat{\gamma}_{1},\dots,\widehat{\gamma}_{n+1}$ on
a table $\widehat{M}_{n}.$ Again, we may assume that $\widehat{\gamma}_{n+1}$
is not perpendicular to the boundary of $\widehat{M}_{n}$ at any
vertex.

We now proceed to achieve conditions (GP3)--(GP4). A segment $\eta$
of $\widehat{\gamma}_{n+1}$ can cause a violation of condition (GP3)
or (GP4) by passing through any of the following (finitely many) points:
(i) an intersection point (other than $x$ or $y$) of two segments
of $\widehat{\gamma}_{1},\dots,\widehat{\gamma}_{n};$ (ii) an intersection
point (other than $x$ or $y$) of one of $\widehat{\gamma}_{1},\dots,\widehat{\gamma}_{n}$
and a segment of $\widehat{\gamma}_{n+1}$ different from $\eta;$
(iii) an intersection point (other than $x$ or $y$) of two distinct
segments of $\widehat{\gamma}_{n+1}$ other than $\eta;$ (iv) $x$
(if $\eta$ is not the initial segment of $\widehat{\gamma}_{n+1});$
and (v) $y$ (if $\eta$ is not the final segment of $\widehat{\gamma}_{n+1}).$
If such a violation occurs, then we can apply Corollary \ref{parallel perturbation}
if $\eta$ is a chord of $\widehat{\gamma}_{n+1}$ and Corollary \ref{perturb initial segment}
if $\eta$ is an initial or final segment of $\widehat{\gamma}_{n+1}$
in order to avoid these violations of conditions (GP3)--(GP4). This
can be done while maintaining conditions (GP1)--(GP2) for $\widehat{\gamma}_{1},\dots,\widehat{\gamma}_{n}$
and the perturbed version of $\widehat{\gamma}_{n+1},$ and avoiding
any new violations of conditions (GP3)--(GP4). This procedure can
be done to each segment of $\widehat{\gamma}_{n+1}$ that causes a
violation of conditions (GP3)--(GP4). We can therefore achieve conditions
(GP3)--(GP4) for a new table $\widehat{\widehat{M}}_{n},$ a perturbed
version $\widehat{\widehat{\gamma}}_{n+1}$ of the $(n+1)$st path,
and the same paths $\widehat{\gamma}_{1},\dots,\widehat{\gamma}_{n}.$ 

Since condition (NC) is already satisfied for $\widehat{\gamma}{}_{1},\dots,\widehat{\gamma}_{n},$
any violation of condition (NC) for the $n+1$ paths $\widehat{\gamma}_{1},\dots,\widehat{\gamma}_{n},\widehat{\widehat{\gamma}}_{n+1}$
on $\widehat{\widehat{M}}_{n}$ would mean there exist distinct vertices
$p,q$ of $\widehat{\gamma}_{1},\dots,\widehat{\gamma}_{n+1}$ that
are collinear with $y$, where $p$ is a vertex of $\widehat{\gamma}_{n+1}$.
Note that the segment $pq$ cannot be a chord of $\widehat{\gamma}_{n+1},$
because then condition (GP4) would be violated. If $p$ is not an
endpoint of the final segment of $\hat{\gamma}_{n+1}$, then we can
apply Corollary \ref{parallel perturbation} to a chord through $p$
to obtain a perturbed version of $\widehat{\gamma}_{n+1}$ so that
the new vertex near $p$ is not collinear with $y$ and $q$. If $p$
is an endpoint of the final segment of $\widehat{\gamma}_{n+1}$,
then we can apply Corollary \ref{perturb initial segment} to achieve
this result. Again, all perturbations can be done while maintaining
conditions (GP1)--(GP4) and not introducing any new violations of
condition (NC). The paths $\widehat{\gamma}_{1},\dots,\widehat{\gamma}_{n}$
and the perturbed version of $\widehat{\widehat{\gamma}}_{n+1}$ are
then labeled $\gamma_{1},\dots,\gamma_{n+1}$ and the perturbed version
of $\widehat{\widehat{M}}_{n}$ is labeled $M_{n+1}.$ Thus conditions
(GP1)--(GP4) and (NC) are satisfied for $\gamma_{1},\dots,\gamma_{n+1}$
on $M_{n+1}.$

The boundary $\sigma_{n+1}$ of $M_{n+1}$ is in $\Sigma_{r}$ and
can be chosen to satisfy $d_{r}(\sigma_{n},\sigma_{n+1})<\epsilon_{n}$
since we made arbitrarily small $C^{r}$ perturbations of the boundaries
of the tables at each step in the proof.
\end{proof}

\section{The Main Result: Generic Insecurity }

Let $\Sigma_{r}$ and $d_{r}$ be as described in Section 3, and assume
$r\ge2$. Then $\Sigma_{r}$ is a dense open subset of the complete
metric space $(\Sigma_{r}^{0},d_{r}),$ where $\Sigma_{r}^{0}$ is
the set of simple closed $C^{r}$ curves in the plane with nonnegative
curvature. Thus, by the Baire Category Theorem, in the $C^{r}$ topology
the intersection of countably many dense open subsets of $\Sigma_{r}$
is dense in $\Sigma_{r}.$ Theorem \ref{thm:main result} contains
the generic insecurity result described in Section 1. In fact, we
prove slightly more: the set $A$ in Theorem \ref{thm:main result}
is not just a dense $G_{\delta}$ subset of $\Sigma_{r}$ in the $C^{r}$
topology---it is the intersection of countably many sets that are
$C^{2}$ open and $C^{r}$ dense in $\Sigma_{r}.$
\begin{lem}
\label{lemma open set} Let $\sigma_{n}\in\Sigma_{2}$ and suppose
$x$ and $y$ are distinct points in the interior of the table $M_{n}$
bounded by $\sigma_{n}.$ Assume there exist $n$ billiard paths $\gamma_{1},\ldots\gamma_{n}$
for $M_{n}$ from $x$ to $y$ that are in general position, and $x$
and $y$ are not conjugate along any of these paths. Then there exists
an open neighborhood $\mathcal{N}$ of $\sigma_{n}$ in $\Sigma_{2}$
with the $C^{2}$ topology such that for every $\widehat{\sigma}\in\mathcal{N}$,
$x$ and $y$ are still in the interior of the table $\widehat{M}$
bounded by $\widehat{\sigma}$, and there exist $n$ billiard paths
$\widehat{\gamma}_{1},\dots,\widehat{\gamma}_{n}$ for $\widehat{M}$
from $x$ to $y$ that are in general position.\end{lem}
\begin{proof}
Let $\epsilon>0$ be as in Lemma \ref{lem:preservation of general position}.
Then it suffices to show that there exists a $C^{2}$ open neighborhood
$\mathcal{N}$ of $\sigma_{n}$ in $\Sigma_{2}$ such that for all
$\widehat{\sigma}\in\mathcal{N},$ there exist billiard paths $\widehat{\gamma}_{1},\dots,\widehat{\gamma}_{n}$
for the table $\widehat{M}$ bounded by $\widehat{\sigma}$ from $x$
to $y$ such that $\widehat{\gamma}_{i}$ has the same number of vertices
as $\gamma_{i}$ and $d(\gamma_{i},\widehat{\gamma}_{i})<\epsilon$
for $i=1,\dots,n.$ Fix a choice of $i\in\{1,\dots,n\},$ and let
$k_{i}$ be the number of vertices of $\gamma_{i}.$ It follows from
Lemma \ref{perturbed} that there exists $\alpha_{i}>0$ and a $C^{2}$
open neighborhood $\widehat{\mathcal{N}}_{i}$ of $\sigma_{n}$ such
that for any billiard path $\widehat{\gamma}_{i}$ for any $\widehat{\sigma}\in\widehat{\mathcal{N}}_{i}$
that starts at $x$ making angle less than $\alpha_{i}$ with $\gamma_{i}$
at $x,$ the first $k_{i}$ vertices of $\widehat{\gamma}_{i}$ are
within distance $\epsilon$ of the corresponding vertices of $\gamma_{i}.$
Moreover, by Corollary \ref{cor: Neigborhood filling}, if $\mathcal{N}_{i}\subset\widehat{\mathcal{N}}_{i}$
is a sufficiently small $C^{2}$ open neighborhood of $\sigma_{n},$
then for $\widehat{\sigma}\in\mathcal{N}_{i}$ there exists a billiard
path $\widehat{\gamma}_{i}$ starting at $x$ and making angle less
than $\alpha_{i}$ with $\gamma_{i}$ at $x$ such that $\widehat{\gamma}_{i}$
ends at $y$ after $k_{i}$ reflections. Thus $\mathcal{N}:=\cap_{i=1}^{n}\mathcal{N}_{i}$
is the desired neighborhood of $\sigma_{n}.$ \end{proof}
\begin{cor}
\label{cor:generic table} Given two distinct points $x$ and $y$
in $\mathbb{R}^{2},$ let $\Sigma_{(x,y),r}$ be the set of curves
$\sigma$ in $\Sigma_{r}$ such that $x$ and $y$ are contained in
the interior of the table bounded by $\sigma.$ Then there exists
a dense $G_{\delta}$ subset $A_{(x,y),r}$ of $\Sigma_{(x,y),r}$
in the $C^{r}$ topology such that for every $\tau\in A_{(x,y),r}$
and every positive integer $n$ there exist $n$ billiard paths from
$x$ to $y$ for the table $M(\tau)$ bounded by $\tau$ that are
in general position. In particular, for every $\tau\in A_{(x,y),r},$
the pair $(x,y)$ is insecure for $M(\tau).$ The set $A_{(x,y),r}$
may be chosen to be an intersection of countably many $C^{2}$ open
and $C^{r}$ dense subsets of $\Sigma_{(x,y),r.}$\end{cor}
\begin{proof}
Let $n$ be a positive integer, $\epsilon>0,$ and $\sigma\in\Sigma_{(x,y),r}.$
By Theorem \ref{theorem 1}, there exists $\sigma_{n}\in\Sigma_{(x,y),r}$
with $d_{r}(\sigma,\sigma_{n})<\epsilon$ such that there exist $n$
billiard paths $\gamma_{1},\dots,\gamma_{n}$ from $x$ to $y$ for
the table bounded by $\sigma_{n}$ that are in general position. By
Corollary \ref{nonconjugacy }, we may assume that $x$ and $y$ are
not conjugate along any of $\gamma_{1},\dots,\gamma_{n}$ in the billiard
table bounded by $\sigma_{n}.$ Now by Lemma \ref{lemma open set}
there exists a $C^{2}$ neighborhood $\mathcal{N}=\mathcal{N}(\sigma,n,\epsilon)$
of $\sigma_{n}$ in $\Sigma_{(x,y),r}$ such that for every $\widehat{\sigma}\in\mathcal{N}$
 there are $n$ billiard paths for $M(\widehat{\sigma})$ from $x$
to $y$ that are in general position. Then $A_{(x,y),r}:=\cap_{n=1}^{\infty}\cup_{\epsilon>0}\cup_{\sigma\in\Sigma_{(x,y),r}}\mathcal{N}(\sigma,n,\epsilon)$
is a dense $G_{\delta}$ subset of $\Sigma_{(x,y),r}$ satisfying
the conclusion of the Corollary.
\end{proof}
For $k=2$ or $r,$ we refer to the product topology on $\Sigma_{r}^{0}\times\mathbb{R}^{2}\times\mathbb{R}^{2}$
obtained from the $C^{k}$ topology on $\Sigma_{r}^{0}$ and the usual
topology on $\mathbb{R}^{2}\times\mathbb{R}^{2}$ as the $C^{k}$
topology. The $C^{r}$ topology on $\Sigma_{r}^{0}\times\mathbb{R}^{2}\times\mathbb{R}^{2}$
is the topology of a complete metric space. 
\begin{thm}
\label{thm:main result}For $r\ge2$ there exists a dense $G_{\delta}$
subset $A$ of $\Sigma_{r}$ in the $C^{r}$ topology such that if
$\sigma\in A$ and $M(\sigma)$ is the billiard table bounded by $\sigma,$
then there is a dense $G_{\delta}$ subset $\mathcal{R}(\sigma)$
of $M(\sigma)\times M(\sigma)$ with the topology induced from $\mathbb{R}^{2}\times\mathbb{R}^{2}$
such that for each $(x,y)\in\mathcal{R}(\sigma)$ and each positive
integer $n,$ there are $n$ billiard paths for $M(\sigma)$ from
$x$ to $y$ that are in general position. In particular, every pair
$(x,y)\in\mathcal{R}(\sigma)$ is insecure for $M(\sigma).$ The set
$A$ may be chosen to be an intersection of countably many $C^{2}$
open and $C^{r}$ dense subsets of $\Sigma_{r}.$ \end{thm}
\begin{proof}
Consider $\Sigma_{r}\times B,$ where $B=\{(x,y)\in\mathbb{R}^{2}\times\mathbb{R}^{2}:x\ne y\},$
with the $C^{r}$ topology induced from $\Sigma_{r}^{0}\times\mathbb{R}^{2}\times\mathbb{R}^{2}.$
Let $\mathcal{G}=\{(\sigma,(x,y))\in\Sigma_{r}\times B:x{\rm \ and\ }y{\rm \ do\ not\ lie\ on\ }\sigma\}$
and let $\mathcal{G}_{0}=\{(\sigma,(x,y))\in\Sigma_{r}\times B:(x,y)\in{\rm Int}(M(\sigma))\times{\rm Int}(M(\sigma))\}.$
Then $\mathcal{G}$ is a $C^{2}$ open and $C^{r}$ dense subset of
$\Sigma_{r}^{0}\times\mathbb{R}^{2}\times\mathbb{R}^{2},$ and $\mathcal{G\setminus\mathcal{G}}_{0}$
is $C^{2}$ open. If $(\sigma,(x,y))\in\mathcal{G}_{0},$ then as
in the proof of Corollary \ref{cor:generic table}, if we are given
$\epsilon>0$ and a positive integer $n,$ there exist $\sigma_{n}\in\Sigma_{r}$
with $d_{r}(\sigma,\sigma_{n})<\epsilon$ and a $C^{2}$ open neighborhood
$\mathcal{\mathcal{N}=N}(\sigma,n,\epsilon,(x,y))$ of $\sigma_{n}$
in $\Sigma_{r}$ such that for every $\widehat{\sigma}\in\mathcal{N}$
there exist $n$ billiard paths from $x$ to $y$ for $M(\widehat{\sigma})$
that are in general position. Note that if we make a sufficiently
small perturbation of $(x,y)$ to $(\widehat{x},\widehat{y})$ in
addition to a sufficiently small $C^{2}$ perturbation of $\sigma_{n}$
to $\widetilde{\sigma}$$,$ we will still obtain $n$ billiard paths
from $\widehat{x}$ to $\widehat{y}$ in $M(\widetilde{\sigma})$
that are in general position, because we may apply a transformation
$T$ of $\mathbb{R}^{2}$ such that $T$ is the composition of a homothety
and a rigid motion, both close to the identity, such that $T$ maps
$\widehat{x}$ and $\widehat{y}$ to $x$ and $y,$ respectively,
and maps $\widetilde{\sigma}$ to a curve in $\mathcal{N}.$ Thus
there is a $C^{2}$ open neighborhood $\widetilde{\mathcal{N}}=\widetilde{\mathcal{N}}(\sigma,n,\epsilon,(x,y))$
of $(\sigma_{n},(x,y))$ in $\mathcal{G}_{0}$ such that for each
$(\widehat{\sigma},(\widehat{x},\widehat{y}))\in\mbox{\ensuremath{\widetilde{\mathcal{N}}}}$
there are $n$ billiard paths for $M(\widehat{\sigma})$ from $\widehat{x}$
to $\widehat{y}$ that are in general position. Let $\mathcal{G}_{n}=(\mathcal{G}\setminus\mathcal{G}_{0})\cup(\cup_{\epsilon>0}\cup_{(\sigma,(x,y))\in\mathcal{G}_{0}}\widetilde{\mathcal{N}}(\sigma,n,\epsilon,(x,y))$.
Then $\mathcal{G}_{n}$ is a $C^{2}$ open and $C^{r}$ dense subset
of $\mathcal{G}.$ Thus there exists a subset $A_{n}$ of $\Sigma_{r}$
which is the intersection of countably many $C^{2}$ open and $C^{r}$
dense subsets of $\Sigma_{r}$ such that for $\sigma\in A_{n},$ $\{(x,y)\in\mathbb{R}^{2}\times\mathbb{R}^{2}:(\sigma,(x,y))\in\mathcal{G}_{n}\}$
is a dense $G_{\delta}$ subset of $\mathbb{R}^{2}\times\mathbb{R}^{2}.$
(This follows from the proof of the Kuratowski-Ulam Theorem, as presented
in Chapter 15 of \cite{Oxt96}.) Let $A=\cap_{n=1}^{\infty}A_{n}.$
If $\sigma\in A,$ then $\mathcal{R}(\sigma):=\cap_{n=1}^{\infty}\{(x,y)\in\mathbb{R}^{2}\times\mathbb{R}^{2}:(\sigma,(x,y))\in\mathcal{G}_{n}\setminus\mathcal{G}_{0}\}$
is a dense $G_{\delta}$ subset of ${\rm Int}(M(\sigma))\times{\rm Int}(M(\sigma))$
having the required property.
\end{proof}

\section*{Acknowledgment}

We thank Sergei Tabachnikov for a helpful conversation concerning
this paper.

{\medskip\smaller\scshape Department of Mathematics, Indiana University, Bloomington, IN 47405, USA}\\
\emph{E-mail address}: {\tt tjdauer@indiana.edu, gerber@indiana.edu}

\begin{figure}
	\centering
\begin{tikzpicture}[line cap=round,line join=round,>=triangle 45,x=1.0cm,y=1.0cm] 		\clip(4.77445412898742,-4.541816658119153) rectangle (19.109468585764564,3.849229103360718);
		\draw [rotate around={0.:(11.11,-0.32)}] (11.11,-0.32) ellipse (5.567513548539673cm and 3.8203019662289623cm); 		\draw (7.06,-0.32)-- (12.158932151223185,3.431887786368514); 		\draw (12.158932151223185,3.431887786368514)-- (14.78539868402223,0.14832066884038686); 		\draw [shift={(10.962116349441303,-10.193134356136868)},dash pattern=on 1pt off 1pt]  plot[domain=1.3560879464883926:1.7368992278911963,variable=\t]({1.*13.512471182452135*cos(\t r)+0.*13.512471182452135*sin(\t r)},{0.*13.512471182452135*cos(\t r)+1.*13.512471182452135*sin(\t r)}); 		\draw [dash pattern=on 1pt off 1pt] (7.06,-0.32)-- (12.302247433103593,3.252717276809314); 		
{\tiny		
\draw [fill=black] (7.06,-0.32) circle (1.5pt); 		\draw[color=black] (6.871696382561103,-0.1469815265127122) node {$A$}; 		\draw[color=black] (11,3.6) node {$\sigma$}; 		\draw [fill=black] (12.158932151223185,3.431887786368514) circle (1.5pt); 		\draw[color=black] (12.23319593087011,3.7) node {$P$}; 		\draw [fill=black] (12.158932151223185,3.431887786368509) circle (1.5pt); 		\draw [fill=black] (12.158932151223183,3.431887786368514) circle (1.5pt); 		\draw [fill=black] (14.78539868402223,0.14832066884038686) circle (1.5pt); 		\draw[color=black] (15.1,0.2715742003069488) node {$A_1$}; 		\draw[color=black] (11,3.05) node {$\widetilde{\sigma}$}; 		\draw [fill=black] (12.302247433103593,3.252717276809314) circle (1.5pt); 		\draw[color=black] (12.283023993586736,2.95) node {$\widetilde{P}$}; 		
}
\end{tikzpicture} 	\centerline{Figure 1. Changing the direction of a billiard path at $A$.} \end{figure}

\begin{picture}(0,0)%
\includegraphics{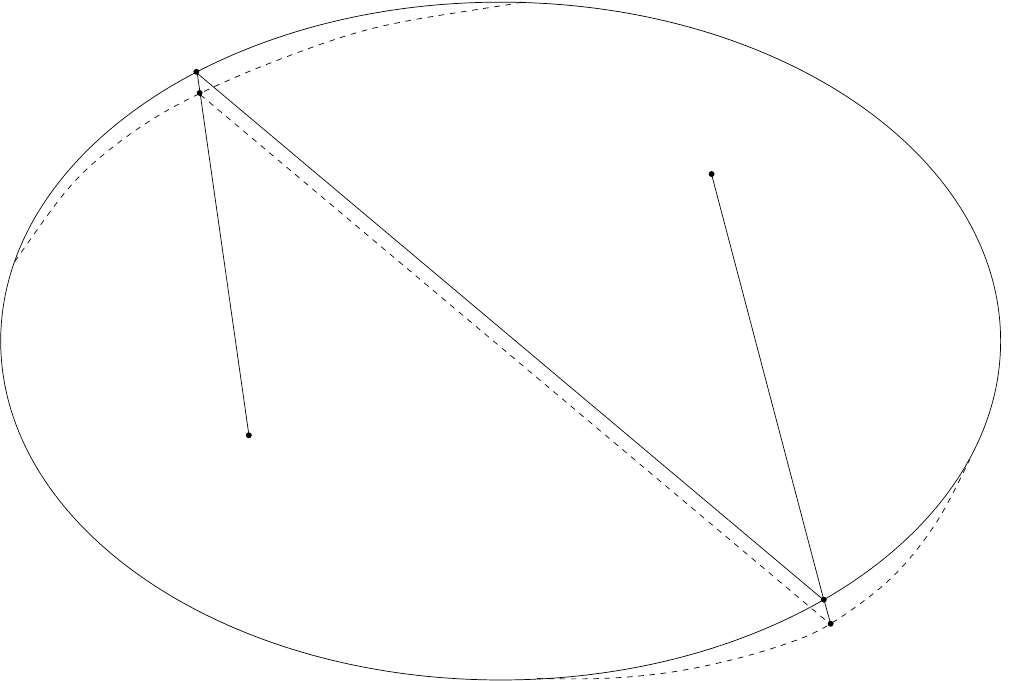}%
\end{picture}%
\setlength{\unitlength}{1776sp}%
\begingroup\makeatletter\ifx\SetFigFont\undefined%
\gdef\SetFigFont#1#2#3#4#5{%
  \reset@font\fontsize{#1}{#2pt}%
  \fontfamily{#3}\fontseries{#4}\fontshape{#5}%
  \selectfont}%
\fi\endgroup%
\centering
\begin{picture}(10783,7252)(10930,7927)

\put(13450,10200){\makebox(0,0)[lb]{\smash{{\SetFigFont{8}{6.0}{\familydefault}{\mddefault}{\updefault}{$A$}%
}}}}
\put(18400,13500){\makebox(0,0)[lb]{\smash{{\SetFigFont{8}{6.0}{\familydefault}{\mddefault}{\updefault}{$B$}%
}}}}
\put(19700,8100){\makebox(0,0)[lb]{\smash{{\SetFigFont{8}{6.0}{\familydefault}{\mddefault}{\updefault}{$\widetilde{Q}$}%
}}}}
\put(19700,9050){\makebox(0,0)[lb]{\smash{{\SetFigFont{8}{6.0}{\familydefault}{\mddefault}{\updefault}{$Q$}%
}}}}
\put(12900,14600){\makebox(0,0)[lb]{\smash{{\SetFigFont{8}{6.0}{\familydefault}{\mddefault}{\updefault}{$P$}%
}}}}
\put(12800,13740){\makebox(0,0)[lb]{\smash{{\SetFigFont{8}{6.0}{\familydefault}{\mddefault}{\updefault}{$\widetilde{P}$}%
}}}}
\put(15000,15200){\makebox(0,0)[lb]{\smash{{\SetFigFont{8}{6.0}{\familydefault}{\mddefault}{\updefault}{$\sigma$}%
}}}}
\put(15400,14700){\makebox(0,0)[lb]{\smash{{\SetFigFont{8}{6.0}{\familydefault}{\mddefault}{\updefault}{$\widetilde{\sigma}$}%
}}}}
\put(20750,9100){\makebox(0,0)[lb]{\smash{{\SetFigFont{8}{6.0}{\familydefault}{\mddefault}{\updefault}{$\widetilde{\sigma}$}%
}}}}
\put(20400,9440){\makebox(0,0)[lb]{\smash{{\SetFigFont{8}{6.0}{\familydefault}{\mddefault}{\updefault}{$\sigma$}%
}}}}
\end{picture}%

\medskip\centerline{Figure 2. Replacing a segment of a billiard path by a parallel segment.}
 
\end{document}